\documentclass[reqno]{amsproc}
\usepackage[utf8]{inputenc}
\usepackage{orcidlink}
\usepackage{amsfonts}
\usepackage{graphicx}
\usepackage{amscd}
\usepackage{url}
\usepackage[normalem]{ulem}
\usepackage{lmodern}
\usepackage{amsmath,amssymb,amsthm}
\usepackage{latexsym}
\usepackage{hyperref}
\usepackage{xcolor}
\usepackage[all]{xy}
\usepackage{color}
\usepackage{mathrsfs}
\makeatletter
\patchcmd\@thm{\let\thm@indent\indent}{\let\thm@indent\noindent}%
  {}{}
\makeatother
\theoremstyle{plain}
\newtheorem{theorem}{\bf Theorem}
\newtheorem{lemma}{\bf Lemma}

\newtheorem{proposition}{\bf Proposition}

\theoremstyle{definition}
\newtheorem{remark}{\bf Remark}
\newtheorem{example}{\bf Example}

\usepackage{amsmath,amssymb,amsthm}
\usepackage[english]{babel}
\usepackage{enumerate}
\hypersetup{
  colorlinks  = true,
  linkcolor   = blue,
  anchorcolor = red,
  citecolor   = blue,
  filecolor   = red,
  urlcolor    = red,
  pdfauthor   = author
}
\usepackage{mathtools}
\usepackage{xcolor}
\usepackage{xparse}
\DeclarePairedDelimiter\norm{\lVert}{\rVert}

\NewDocumentCommand{\gradient}{m d<>}{
  \IfNoValueTF{#2}{
    \nabla{#1}
  }{
    \nabla_{#2}{#1}
  }
}

\NewDocumentCommand{\hessian}{m d<> d() d[]}{
  \IfNoValueTF{#2}{
    \nabla{d{#1}}
  }{
    \nabla_{#2}{d{#1}}
  }
  \IfValueTF{#3}{
    \IfValueTF{#4}{
      \left({#3},{#4}\right)
    }{
      \left({#3},\cdot\right)
    }
  }{
    \IfValueT{#4}{
      \left(\cdot,{#4}\right)
    }
  }
}

\NewDocumentCommand{\laplacian}{m d<>}{
  \IfNoValueTF{#2}{
    \Delta{#1}
  }{
    \left(\Delta{#1}\right)_{#2}
  }
}

\NewDocumentCommand{\driftedlaplacian}{m m d<>}{
  \IfNoValueTF{#3}{
    \Delta_{#2}{#1}
  }{
    \left(\Delta_{#2}{#1}\right)_{#3}
  }
}

\NewDocumentCommand{\riccitensor}{d<> d() d[]}{
  \IfValueTF{#1}{
    \textup{Ric}_{#1}
  }{
    \textup{Ric}
  }
  \IfValueTF{#2}{
    \IfValueTF{#3}{
      \left({#2},{#3}\right)
    }{
      \left({#2},\cdot\right)
    }
  }{
    \IfValueT{#3}{
      \left(\cdot,{#3}\right)
    }
  }
}

\NewDocumentCommand{\scalarcurvature}{d<>}{
  \IfNoValueTF{#1}{
    R
  }{
    R_{#1}
  }
}

\NewDocumentCommand{\bakryemerytensor}{m m d<> d() d[]}{
  \IfNoValueTF{#3}{
    \textup{Rc}^{{#1},{#2}}
  }{
    \textup{Rc}^{{#1},{#2}}_{#3}
  }
  \IfValueTF{#4}{
    \IfValueTF{#5}{
      \left({#4},{#5}\right)
    }{
      \left({#4},\cdot\right)
    }
  }{
    \IfValueT{#5}{
      \left(\cdot,{#5}\right)
    }
  }
}
\begin{document}
\title[Gradient estimates for a parabolic PDE under the RBF]{Gradient estimates for a parabolic partial differential equation under the Ricci-Bourguignon flow}

\author[José N.V. Gomes]{José N.V. Gomes\orcidlink{0000-0001-5678-4789}}
\author[Willian I. Tokura]{Willian I. Tokura\orcidlink{0000-0001-9363-793X}}
\author[Hikaru Yamamoto]{Hikaru Yamamoto\orcidlink{0000-0002-7047-2914}}

\address{Departamento de Matemática, Universidade Federal de São Carlos, Rod. Washington Luís, Km 235, 13565-905, São Carlos, São Paulo, Brazil}
\email{jnvgomes@ufscar.br}
\urladdr{https://www.ufscar.br}

\address{Faculdade de Ciências Exatas e Tecnologia, Universidade Federal da Grande Dourados, Mato Grosso do Sul, Brazil.}
\email{williantokura@ufgd.edu.br}
\urladdr{https://portal.ufgd.edu.br}

\address{Department of Mathematics, Faculty of Pure and Applied Science, University of Tsukuba, 1-1-1 Tennodai, Tsukuba, Ibaraki 305-8571, Japan.}
\email{hyamamoto@math.tsukuba.ac.jp}
\urladdr{https://nc.math.tsukuba.ac.jp}


\begin{abstract} 
We study the Ricci-Bourguignon flow on warped product manifolds with noncompact base. This setting leads naturally to a parabolic partial differential equation on the space of smooth warping functions, arising from the necessary and sufficient conditions for a warped metric to evolve under the flow. One of our main results establishes a gradient estimate for this equation, providing the analytic input for the geometric applications developed herein and, in particular, recovering classical gradient estimates for the heat equation under the Ricci flow. Furthermore, we develop a method for constructing explicit warped product solutions to the Ricci-Bourguignon flow and present examples that illustrate the scope and geometric relevance of our results.
\end{abstract}
\maketitle

\section{Introduction and statements of the main results}\label{intro}

A one-parameter family of Riemannian metrics $\overline{g}(t)$ on an $n$-dimensional smooth manifold $M^n$ is a solution to the Ricci-Bourguignon flow if it satisfies the evolution equation
\begin{equation}\label{Eq-RBF}
\frac{\partial}{\partial  t} \overline{g}(t)=-2\big(\textup{Ric}_{\overline{g}(t)}-\rho S_{\overline{g}(t)}\overline{g}(t)\big),\quad \overline{g}(0)=\overline{g}_0,
\end{equation}
for an initial Riemannian metric $\overline{g}_0$ and some real constant $\rho$, where $\textup{Ric}_{\overline{g}(t)}$ and $S_{\overline{g}(t)}$ are the Ricci curvature tensor and the scalar curvature of the Riemannian manifold $(M^n,\overline{g}(t))$, respectively.

The Ricci-Bourguignon flow is only weakly parabolic for $\rho < 1/2(n-1)$. Nevertheless, by DeTurck’s trick~\cite{deturck1983deforming}, the Ricci-Bourguignon flow is equivalent to a strictly parabolic flow, in the sense that their solutions are equivalent modulo time-dependent diffeomorphisms. This approach has been used by G. Catino et al.~\cite{catino2017ricci} to prove the existence of a short-time solution for \eqref{Eq-RBF}. DeTurck’s trick applied to~\eqref{Eq-RBF} provides
\begin{equation}\label{Eq-RBFD}
\frac{\partial}{\partial  t} \overline{g}(t)=-2\big(\textup{Ric}_{\overline{g}(t)}+\nabla_{\overline{g}(t)}^2\varphi-\rho S_{\overline{g}(t)}\overline{g}(t)\big),
\end{equation}
where $\nabla_{\overline{g}(t)}^2\varphi$ denotes the Hessian of $\varphi(\cdot,t)$ taken with respect to the Levi-Civita connection of $\bar{g}(t)$, for each fixed $t$. For convenience, we will refer to the resulting flow~\eqref{Eq-RBFD} as the Ricci-Bourguignon flow.

In the special case of warped metrics, the Ricci-Bourguignon flow is naturally associated with an interesting parabolic differential equation on the space of smooth warping functions defined on the base of a product manifold, as we shall describe now. Let $B^n \times F^m$ be a Riemannian product manifold, where $(F^m,g_F)$ is an $m$-dimensional Einstein manifold with scalar curvature $S_F$.  Given a one-parameter family of Riemannian metrics $g(t)$ and a family of positive smooth warping functions $f(t)$, both defined on $B^n$, we consider the family of warped metrics
\begin{equation}\label{Eq:warped-metric}
\overline{g}(t)=g(t)+f(t)^2 g_F,
\end{equation}
evolving according to~\eqref{Eq-RBFD}. To describe the evolution of the warping function, set
\begin{equation*}
u(x,t)=f(x,t)^{\frac{1}{\sigma}},
\qquad \mbox{where}\qquad
\sigma=\frac{1-2m\rho}{m-\rho m^2-m\rho}.
\end{equation*}
Combining~\eqref{Eq-RBFD} with \eqref{Eq:warped-metric}, we obtain that $u$ satisfies the following parabolic equation on $B^n:$
\begin{equation}\label{PDE-parabolic}
\dfrac{\partial u}{\partial t}(x,t)=\Delta_{\varphi} u(x,t)-2m\rho\Delta u(x,t)+\frac{\rho}{\sigma}S_{g(t)} u(x,t)+\frac{m\rho-1}{m\sigma}S_{F}u(x,t)^{1-2\sigma},
\end{equation}
where $S_{g(t)}$ denotes the scalar curvature of $g(t)$ and $\Delta_{\varphi}$ is the
drifted Laplacian associated with $\varphi$. It is worth emphasizing that equation~\eqref{PDE-parabolic} arises naturally from the necessary and sufficient conditions for constructing warped solutions $\bar{g}(t)$ to the Ricci-Bourguignon flow (see Theorem~\ref{thm1}). 

\begin{remark}
Equation~\eqref{PDE-parabolic} can be viewed as a time-dependent version of equation~(2) in~\cite{gomes2025gradient}, which arises in the study of $\rho$-Einstein soliton warped products. Moreover, as Theorem~\ref{thm1} makes clear, the base manifold provides the natural setting for studying solutions to the Ricci-Bourguignon flow on warped products. Indeed, the evolution of the full system on $(B^n\times F^m,\overline{g}(t))$ is equivalent to a coupled evolution system on $(B^n,g(t))$. In particular, any solution of~\eqref{PDE-parabolic} on the base admits a canonical lift to the product. More precisely, if $\pi:B^n\times F^m\to B^n$ denotes the canonical projection and $\tilde{u}=u\circ\pi$, then for every $p=(x,y)\in B^{n}\times F^{m}$ and time $t$, one has
\begin{equation}\label{EQ:WP}
\dfrac{\partial \tilde{u}}{\partial t}(p,t)=\Delta_{\psi}\tilde{u}(p,t)-2m\rho\Delta\tilde{u}(p,t)+\dfrac{\rho}{\sigma} \widetilde{S_{g(t)}}\tilde{u}(p,t)+\dfrac{m\rho-1}{m\sigma}S_{F
}\tilde{u}(p,t)^{1-2\sigma},
\end{equation}
where 
$\psi:=(1-2m\rho)m\sigma\ln \tilde{u}+\tilde{\varphi}$.
\end{remark}

R. Hamilton~\cite{hamilton1995formation} studied the special case of the Ricci flow coupled with the harmonic map heat flow. The coupling introduces a heat-type diffusion term, which improves the parabolic regularization and leads to finer curvature control. This diffusion term also appears in the warped setting considered here. For instance, if the fiber $F^{m}$ is Ricci-flat (or $m=1$), then the curvature term $S_{F}$ vanishes and, from \eqref{PDE-parabolic}, the positive function $u=f^{m}$ satisfies the classical linear heat equation
\begin{equation}\label{Heat-Eq} 
\dfrac{\partial u}{\partial t}(x,t)=\Delta u(x,t).
\end{equation}
Hence, the deformation of the base geometry under the Ricci flow naturally induces a diffusive mechanism on the warping function. In fact, interpreting $(B^n,g(0))$ as a medium with initial temperature $u(\cdot,0)$, the solution of the heat equation gives $u(x,t)$ as the temperature at $x$ and time $t$ while the metric $g(t)$ evolves along the Ricci flow. This physical interpretation has already been observed by M.~Bailesteanu, X.~Cao and A.~Pulemotov~\cite{bailesteanu2010gradient}, who established a gradient estimate for positive solutions to the heat equation evolving under the Ricci flow.

In the case of the Ricci-Bourguignon flow~\eqref{Eq-RBF}, it follows from~\eqref{PDE-parabolic} that the positive function $u=f^{\frac{1}{\sigma}}$ evolves according to the parabolic Yamabe-type equation
\begin{equation}\label{EQ:YAMABE}
\frac{\partial u}{\partial t}(x,t)=(1-2m\rho)\Delta u(x,t)+\frac{\rho}{\sigma} S_{g(t)}u(x,t)+\frac{m\rho-1}{m\sigma}S_{F
}u(x,t)^{1-2\sigma},
\end{equation}
in which the exponent $1-2\sigma$ of the nonlinear term is related to the Yamabe critical Sobolev exponent, and thus reflects the conformal invariance mechanism underlying the classical Yamabe equation. From this perspective, \eqref{EQ:YAMABE} represents a parabolic warped product formulation of the classical prescribed scalar curvature problem. For further details, see F.~Dobarro and E.~Dozo~\cite{dobarro1987scalar}.

The novelty of~\eqref{PDE-parabolic} stems from the wide range of behaviors described by the exponent $1-2\sigma$, which may take arbitrary real values as $\rho$ varies. Consequently, the equation encompasses superlinear ($\sigma<0$), sublinear ($0<\sigma<1/2$), and even singular regimes ($\sigma>1/2$), whose behavior differs substantially and, outside the superlinear Yamabe case, remains largely unexplored.

With the framework established in Sections~\ref{NSCWMSRBF} and \ref{CWMSRBF}, we prove a gradient estimate for positive smooth solutions $u(x,t)$ of a unified parabolic equation that includes \eqref{PDE-parabolic}, \eqref{EQ:WP}, \eqref{Heat-Eq} and \eqref{EQ:YAMABE} as particular cases. 

\begin{theorem}\label{main} Let
$g(t)$, $t\in[t_{0}-T,t_{0}]$, be a family of complete Riemannian metrics on a smooth manifold $M^{n}$ satisfying 
\begin{equation*}
(1-a)\textup{Ric}_{g(t)}+\nabla^{2}_{g(t)}\varphi\geqslant-k_{1}g(t),\quad \frac{\partial}{\partial t} g(t)\geqslant-2k_{2}g(t),\quad k_{1}, k_{2}\geqslant0,\quad 1>a,
\end{equation*}
on the set 
\begin{equation*}
\mathcal{Q}_{R, T}:=\left\{(x, t) \in M \times[t_{0}-T,t_{0}] \mid r\left(x, t\right) \leqslant R\right\},\qquad R\geqslant2,
\end{equation*}
where $r(x,t)$ denotes the distance function from $x_{0}$ with respect to $g(t)$. Let $u(x,t)$ be a positive solution to 
\begin{equation}\label{Eq-Parabolic*}
\frac{\partial u}{\partial t}(x,t)=\Delta_{\varphi} u(x,t)-a\Delta u(x,t)+b(x,t)u(x,t)+cu(x,t)^{\alpha},
\end{equation}
with $0<u<D$ on $\mathcal{Q}_{R, T}$. Then, there exists a constant $C$ such that
\begin{equation}\label{gradient-estimates}
\begin{split}
&\|\nabla \ln u\|\leqslant C\!\left(q-\ln u^{p}\right)\Bigg{\{}k_{1}^{\frac{1}{2}}+k_{2}^{\frac{1}{2}}+\dfrac{1}{R}+\dfrac{1}{(t-t_{0}+T)^{\frac{1}{2}}}+\frac{[(\Gamma_{\overline{\varphi}})^{+}]^{\frac{1}{2}}}{R^{\frac{1}{2}}}\\
&\quad+\sup_{\mathcal{Q}_{R,T}}\|\nabla b\|^{\frac{1}{3}}+\left(\sup_{\mathcal{Q}_{R,T}}b^{+}\right)^{\frac{1}{2}}\!\!+\sup_{\mathcal{Q}_{R,T}}\left\{\left[\left(\alpha-1+\frac{p}{q-\ln u^{p}}\right)c \right]^{+}\right\}^{\frac{1}{2}}u^{\frac{\alpha-1}{2}}\Bigg{\}}
\end{split}
\end{equation}
on $\mathcal{Q}_{R/2,T}$, where $q$ and $p>0$ are constants chosen so that $q-\ln u^{p}\geqslant\delta>0$ for some constant $\delta$, $s^{+}=\max\{s,0\}$ and 
\begin{equation*}
\Gamma_{\overline{\varphi}}=\sup_{(x, t)\in\mathcal{Q}_{R, T}}\left\{\Delta_{\overline{\varphi}} r(x, t) \mid  r(x,t)=1\right\},\qquad \overline{\varphi}=\frac{\varphi}{1-a}.
\end{equation*}
\end{theorem}

\begin{remark}
Theorem~\ref{main} extends the method of~\cite{li1986parabolic,souplet2006sharp, wu2019gradient} to the Bakry-Émery Ricci tensor and the associated drifted Laplacian. For instance, taking $\varphi$ to be constant, $p=1$, $q=1+\ln(D)$, $a=0$, $\alpha=1$, $b\equiv0$ and $c=0$, we recover Souplet-Zhang’s gradient estimate for the heat equation in P.~Souplet and Q.~Zhang~\cite[Theorem~1.1]{souplet2006sharp}. The development of this approach goes back to the seminal work of P. Li and S. Yau~\cite{li1986parabolic}, who established a gradient estimate via the maximum principle on Riemannian manifolds with Ricci curvature bounded from below. J.-Y. Wu~\cite[Theorem~1.1]{wu2019gradient} provides a similar result with $p=1$, $q=1+\ln(D)$, $a=0$, $\delta=1$ and $\alpha\in\mathbb{R}$. The novelty of Theorem~\ref{main} lies in the fact that a suitable choice of the parameters $p$ and $q$ allows us to treat the regime $\alpha\in\mathbb{R}$ while keeping $c$ constant in the Liouville-type arguments developed below. This flexibility is essential in our applications because, depending on the parameter $\rho$ in the Ricci-Bourguignon flow, the nonlinear exponent $\alpha$ can assume any real value.
\end{remark}

We apply Theorem~\ref{main} for studying warped product solutions to the Ricci-Bourguignon flow in long-time regimes. A solution $\bar g(t)$ of~\eqref{Eq-RBFD} is called \emph{ancient} if it is defined on $(-\infty,T)$ for some $T\in\mathbb{R}$, \emph{immortal} if it is defined on $(T,\infty)$ for some $T\in\mathbb{R}$, and \emph{eternal} if it is defined for all real times. 

The first application of Theorem~\ref{main} shows that, under the curvature bounds and the stated monotonicity and growth assumptions, the evolution remains within the standard product geometries. Precisely, the solution is a usual Riemannian product for the entire time interval on which it is defined, so the only long-time models allowed in this regime are product manifold geometries (e.g., cylindrical models).

\begin{theorem}\label{theorem**}
Let $\big(B^{n},g(t)\big)$, $t\in I$, be a family of complete Riemannian metrics on a noncompact manifold $B^n$ with
\begin{equation*}
(1-2m\rho)\textup{Ric}_{g(t)}+\nabla^{2}_{g(t)}\varphi \geqslant 0,\quad \|\nabla \varphi\|_{g(t)}\leqslant C_{0},
\quad 
\frac{\partial}{\partial  t} g(t)\geqslant0,\quad \mbox{on} \quad B\times I
\end{equation*}
where $1-2m\rho>0$, $C_{0}>0$ and $S_{g(t)}(x)=\beta(t)$ for some $\beta\in C^{\infty}(I)$.
Let $(F^{m},g_{F})$ be an Einstein manifold with scalar curvature $S_{F}$ and assume that the warped metric
\begin{equation*}
\overline g(t)=g(t)+f(t)^{2}g_{F}
\end{equation*}
solves the Ricci-Bourguignon flow~\eqref{Eq-RBFD} and either 
$$S_{F}=0, \quad \rho S_{g(t)}=0, \quad \mbox{or}\quad S_{F}<0,\quad \rho S_{g(t)} \ \ \mbox{is a negative constant}.$$
Then the following assertions hold
\begin{enumerate}[(a)]
\item [\textup{(a)}] If $\overline{g}(t)$ defines an immortal solution with $I=(T,\infty)$ and 
\begin{equation*}
\ln f(x,t)=o\left(r(x,t)^{\frac{1}{2}}+(t-T)^{\frac{1}{2}}\right)
\quad
\mbox{near infinity},
\end{equation*}
where $r(x,t)$ denotes the distance from a fixed point $x_{0}\in B^n$ with respect to $g(t)$, then,
$$
\lim_{t\to\infty}\|\nabla \ln f(x_{0},t)\|_{g(t)}=0.
$$
\item [\textup{(b)}] If $\overline{g}(t)$ defines an ancient or eternal solution and 
\begin{equation*}
\ln f(x,t)=o\left(r(x,t)^{\frac{1}{2}}+|t|^{\frac{1}{2}}\right)
\quad
\mbox{near infinity},
\end{equation*}
then $f$ is constant and $\overline{g}(t)$ must be a Riemannian product metric.
\end{enumerate}
\end{theorem}

To quantify the global curvature scale of a maximal solution and to separate the possible long-time regimes in a scale-invariant way, we set
\begin{equation*}
\textup{K}_{\max}(t):=\sup_{x\in M}\,|\textup{Rm}(x,t)|_{\overline{g}(t)} ,\qquad t\in[0,T),
\end{equation*}
where $[0,T)$ is the maximal interval of existence, with $T\in(0,\infty]$.  
Following R. Hamilton~\cite{hamilton1995formation}, maximal solutions can be classified according to the growth of $\textup{K}_{\max}(t)$ relative to the maximal time of existence $T$. More precisely, each maximal solution belongs to exactly one of the following classes:
\begin{enumerate}[(a)]
\item \textbf{Type~I:} $T<\infty$ and $\displaystyle \sup_{t\in[0,T)} (T-t)\, \textup{K}_{\max}(t)<\infty$.
\smallskip
\item \textbf{Type~II(a):} $T<\infty$ and $\displaystyle \sup_{t\in[0,T)} (T-t)\, \textup{K}_{\max}(t)=\infty$.
\smallskip
\item \textbf{Type~II(b):} $T=\infty$ and $\displaystyle \sup_{t\geqslant 0} t\textup{K}_{\max}(t)=\infty$.
\smallskip
\item \textbf{Type~III:} $T=\infty$ and $\displaystyle \sup_{t\geqslant 0} t\textup{K}_{\max}(t)<\infty$.
\end{enumerate}

Theorem~\ref{theorem**} gives a classification result compatible with Hamilton’s trichotomy. For instance, consider the base $(\mathbb{R}, g(t))$ with $g(t)$ nondecreasing in time and a flat fiber $(F^{m},g_{F})$. Then, for any family of warping functions $f(t)$ satisfying
\begin{equation}\label{growcondition}
\ln f(x,t)=o\left(r(x,t)^{\frac{1}{2}}+|t|^{\frac{1}{2}}\right)\quad \text{near infinity},
\end{equation}
and such that
\begin{equation*}
\overline g(t)=g(t)+f(t)^{2}g_{F}
\end{equation*}
defines an eternal maximal solution to the Ricci flow on $\mathbb{R}\times F^m$, one deduces that $\bar{g}(t)$ must be a maximal solution of Type~III. Now, to obtain a Type~II(b) maximal solution, consider the base $(\mathbb{R},g(t))$ with $g(t)$ nondecreasing in time and let $(F^4,g_{F})$ be the Euclidean Schwarzschild manifold.
If $f(t)$ satisfies 
\begin{equation*}
\ln f(x,t)=o\left(r(x,t)^{\frac{1}{2}}+(t-T)^{\frac{1}{2}}\right)
\quad
\mbox{near infinity},
\end{equation*} 
and $\liminf_{t\to\infty} f(x_{0},t)<\infty$, then any immortal solution $\overline g(t)=g(t)+f(t)^{2}g_{F}$ to the Ricci-Bourguignon flow~\eqref{Eq-RBF} must be a maximal solution of Type~II(b). In fact, Theorem~\ref{theorem**} implies that
$$
\lim_{t\to\infty}\|\nabla \ln f(x_{0},t)\|_{g(t)}=0.
$$
Since $g_{F}$ is Ricci flat and nonflat, there exist $y_{0}\in F$ and a two-plane $\Pi\subset T_{y_{0}}F$ such that $K_{g_{F}}(\Pi)=\kappa>0$. Fix $A>\liminf_{t\to\infty} f(x_{0},t)$. By the definition of the lower limit, there exists a sequence
$t_{j}\to\infty$ such that $f(x_{0},t_{j})\leqslant A$. The sectional curvature formula for warped products yields
$$
K_{\bar{g}(t_j)}(\Pi)=\frac{\kappa}{f(x_0, t_j)^2}-\left\|\nabla \ln f(x_0, t_j)\right\|_{g(t_j)}^2\geqslant\frac{\kappa}{A^2}-\left\|\nabla \ln f(x_0, t_j)\right\|_{g(t_j)}^2
$$
and for all sufficiently large $j$,
$$
K_{\bar{g}\left(t_j\right)}(\Pi) \geq \frac{\kappa}{2 A^2}>0.
$$
Consequently,
$$
\left(t_j-T\right) \textup{K}_{\max }(t_j) \geqslant \frac{\kappa}{2 A^2}\left(t_j-T\right) \longrightarrow \infty .
$$
After translating the initial time $T$ to zero, it follows that $\bar{g}(t)$ is a maximal solution of Type~II(b).

Theorem~\ref{theorem**} also leads to a rigidity result for ancient solutions to the Ricci-Bourguignon flow on warped products. More precisely, if $\mathbb{R}\times_{f}F^{m}$ is an ancient solution with $\rho<\frac{1}{2m}$, a Ricci-flat fiber, and a warping function satisfying \eqref{growcondition}, then the metric must be Ricci-flat. This may be viewed as a complement to \cite[Proposition~4.4]{catino2017ricci}, where it is shown that any ancient solution to the Ricci-Bourguignon flow on a compact manifold $M^n$, with $\rho\leqslant \frac{1}{2(n-1)}$, must either have strictly positive scalar curvature or be Ricci-flat.

\begin{remark} It is worth noting that Theorem~\ref{theorem**} provides a new approach to the study of long-time solutions to the Ricci-Bourguignon flow, differing from those in the literature, where the classification is obtained from assumptions imposed directly on the initial warped product geometry. For instance, T.~Marxen~\cite{marxen2020convergence} proved that any eternal Ricci flow solution of the form $\overline{g}(t)=g(t)+f(t)^{2}g_{F}$ on $\mathbb{R}\times F^{m}$ with compact Ricci-flat fiber must be maximal of Type~III provided a set of conditions on the initial manifold is satisfied. Further examples exhibiting Type~III behavior for the Ricci flow on product manifolds were obtained by J.~Lott and N.~Sesum~\cite{lott2014ricci} in the compact setting, and by T.~Oliynyk and E.~Woolgar~\cite{oliynyk2007rotationally} in the noncompact setting. 
\end{remark}

The second application of Theorem~\ref{main} provides a global obstruction to the existence of warped product solutions. More precisely, we prove the following nonexistence theorem.

\begin{theorem}\label{theorem***}There exists no warped metric on $B^{n}\times F^{m}$ of the form
\begin{equation*}
\overline{g}(t)= g(t)+f(t)^{2}g_{F}
\end{equation*}
that solves the Ricci-Bourguignon flow \eqref{Eq-RBFD}, provided that $g(t)$, $t\in I$, is a family of complete Riemannian metrics on a noncompact smooth manifold $B^n$ satisfying
\begin{equation*}
(1-2m\rho)\textup{Ric}_{g(t)}+\nabla^{2}_{g(t)}\varphi \geqslant 0,\quad \|\nabla \varphi\|_{g(t)}\leqslant C_{0},
\quad 
\frac{\partial}{\partial  t} g(t)\geqslant0,\quad \mbox{on} \quad B\times I,
\end{equation*}
where $1-2m\rho>0$, $C_{0}>0$ and $S_{g(t)}(x)=\beta(t)$ for some $\beta\in C^{\infty}(I)$ and one of the following conditions:
\begin{enumerate}[(a)]
\item [\textup{(a)}] The metric $\overline{g}(t)$ is an ancient or eternal solution, the warping function satisfies
\begin{equation*}
\ln f(x,t)=o\left(r(x,t)^{\frac{1}{2}}+|t|^{\frac{1}{2}}\right)
\quad
\mbox{near infinity}
\end{equation*}
and either 
$$S_{F}<0, \quad \rho S_{g(t)}=0, \quad \mbox{or}\quad S_{F}=0,\quad \rho S_{g(t)} \ \ \mbox{is a negative constant}.$$
\item [\textup{(b)}] The metric $\overline{g}(t)$ is an immortal solution defined on $I=(T,\infty)$, $S_{F}=0$, $\rho S_{g(t)}=0$ and the warping function $f$ is nonconstant and satisfies
\begin{equation*}
 \ln f(x,t)=o\left(r(x,t)^{\frac{1}{2}}+(t-T)^{\frac{1}{2}}\right)
\quad
\mbox{near infinity}
\end{equation*}
and
\begin{equation*}
    \frac{\partial^2}{\partial t^{2}}\left(\|\nabla \ln f(x,t)\|_{g(t)}^{2}\right)\leqslant0\ \ \mbox{on} \ \ B\times I
\end{equation*}
\end{enumerate}
hold.
\end{theorem}

It is worth mentioning that in the particular case of solutions to the Ricci-Bourguignon flow~\eqref{Eq-RBF}, the growth assumption on the warping function $f$ in Theorems~\ref{theorem**} and \ref{theorem***} may be sharpened to
\begin{equation}\label{sharp-condition}
\ln f(x,t)=o\left(r(x,t)+\Theta(t)^{\frac{1}{2}}\right)\quad \mbox{near infinity},
\end{equation}
where $\Theta(t)=|t|$ for ancient and eternal solutions, whereas $\Theta(t)=t-T$ for immortal solutions defined on $(T,\infty)$. In this setting, the operator $\Delta_{\overline{\varphi}}$ reduces to $\Delta$, while the term $\Gamma_{\overline{\varphi}}$ appearing in Theorem~\ref{main} can be removed. Moreover, Example~\ref{Ex2} shows that the condition~\eqref{sharp-condition} is sharp with respect to the spatial factor. 

\section{Necessary and sufficient conditions for constructing warped metric solutions to the Ricci-Bourguignon flow} \label{NSCWMSRBF}

In this section, we establish necessary and sufficient conditions for the construction of warped metric solutions to the Ricci-Bourguignon flow~\eqref{Eq-RBFD}. In the Ricci flow setting, H. Tran~\cite[Lemma~2.1]{tran2016harnack} showed that a warped product structure induces an evolution system on the base manifold. In Theorem~\ref{thm1}, we prove that an analogous result holds for the Ricci-Bourguignon flow and further show that the warped structure forces the fiber to be Einstein. The key step is to prove that the DeTurck term $\varphi(\cdot,t)$ depends only on the base. The proof relies on an auxiliary proposition and a lemma.

\begin{proposition}\label{Proposition-1} Let $(B^{n},g(t))$ be a smooth one-parameter family of complete Riemannian manifolds, $(F^{m},g_{F})$ a complete Riemannian manifold, and $f(t)$ a smooth family of positive nonconstant functions on $B^n$. If the warped metric
\[
\overline{g}(t)=g(t)+f(t)^2g_{F}
\]
is a solution of the Ricci-Bourguignon flow
\begin{equation}\label{RBF-2}
\frac{\partial}{\partial t}\overline{g}(t)
= -2\textup{Ric}_{\overline{g}(t)} - 2\nabla^{2}_{\overline{g}(t)}\varphi
+ 2\rho S_{\overline{g}(t)}\overline{g}(t),
\end{equation}
then $\varphi(\cdot,t)$ depends only on the base $B^n$ and the scalar curvature $S_F$ of the fiber is constant.
\end{proposition}

To prove Proposition~\ref{Proposition-1}, we begin by establishing a general lemma for time-independent warped product Riemannian manifolds.

\begin{lemma}\label{New-lem-1}
Let \((B^{n},g)\) and \((F^{m},g_{F})\) be complete Riemannian manifolds, and \(f\) a smooth positive nonconstant function on \(B^n\). 
Fix \(\rho\in\mathbb{R}\) and  \(\varphi\in C^{\infty}(B^n\times F^m)\). 
For the warped metric
\[
\overline{g}=g+f^2g_{F}, 
\]
put 
\begin{equation}\label{tens-Q}
Q:=\textup{Ric}_{\overline{g}} + \nabla^{2}_{\overline{g}}\varphi
-\rho S_{\overline{g}}\overline{g}. 
\end{equation}
Assume that the following conditions hold:
\begin{enumerate}[(a)]
\item [\textup{(a)}] $Q(X,V)=0$ for $X\in \mathcal{L}(B)$ and $V\in \mathcal{L}(F)$.  
\smallskip
\item [\textup{(b)}] The horizontal component of $Q$, denoted by $Q|_{\mathcal{H}}$, does not depend on the fiber $F^m$. Namely, $V(Q(X,Y))=0$ for all $X,Y\in \mathcal{L}(B)$ and $V\in \mathcal{L}(F)$. 
\smallskip
\item [\textup{(c)}] $\textup{div}_{g_{F}} (Q|_{\mathcal{V}})\!=\!0$, where $Q|_{\mathcal{V}}$ is the vertical component of $Q$ defined by $Q|_{\mathcal{V}}(U,V)=Q(U,V)$ for $U,V\in \mathcal{L}(F)$. 
\smallskip
\item [\textup{(d)}] $\textup{trace}_{g_{F}} (Q|_{\mathcal{V}})$ does not depend on the fiber $F^m$. Namely, $V(\textup{trace}_{g_{F}} (Q|_{\mathcal{V}}))=0$ for all $V\in\mathcal{L}(F)$. 
\end{enumerate}
Then, $\varphi$ depends only on the base $B^n$, and the scalar curvature $S_F$ of the fiber is constant.
\end{lemma}
\begin{proof}
We begin with assumption \textup{(a)}. Let $X\in \mathcal{L}(B)$ and $V\in \mathcal{L}(F)$.  
Since $\overline{g}(X,V)=0$, the scalar curvature term in \eqref{tens-Q} vanishes. 
Moreover, for a warped product, the mixed component of the Ricci tensor vanishes as well. So,
\begin{equation*}
0 =Q(X,V)= \nabla^{2}_{\overline{g}}\varphi(X,V),
\end{equation*}
and it follows that there exist $\eta\in C^{\infty}(B)$ and $\omega\in C^{\infty}(F)$ such that 
\begin{equation*}
\varphi=\eta+f\omega.
\end{equation*}

Let $\mathcal{H}$ and $\mathcal{V}$ denote the horizontal and vertical distributions induced by the warped metric $\bar{g}$, respectively. By the Bishop-O'Neill formulas~\cite{bishop1969}, the horizontal component of the Ricci tensor is given by
\begin{equation*}
\textup{Ric}_{\bar{g}}\Big{|}_{\mathcal{H}}
= \textup{Ric}_{g}-\frac{m}{f}\nabla^{2}_{g}f. 
\end{equation*}
In addition, the scalar curvature formula for warped products yields
\begin{equation*}
\rho S_{\bar{g}}\bar{g}
= \rho\left(
S_{g}
+ \frac{S_{F}}{f^{2}}
- 2m\frac{\Delta f}{f}
- m(m-1)\frac{\|\nabla f\|^{2}}{f^{2}}
\right) \bar{g}.
\end{equation*}

Now, since $\varphi=\eta+f\omega$, for any horizontal vector fields $X,Y\in \mathcal{L}(B)$ we have
\begin{equation}\label{Eq:varphi-Base}
\nabla^{2}_{\overline{g}}\varphi(X,Y)
= \nabla^{2}_{g}\eta(X,Y)
+ \omega\nabla^{2}_{g}f(X,Y).
\end{equation}
Thus, one has
\begin{equation}\label{EQ:Geodesic-Base}
\begin{split}
Q(X,Y)& = \mathrm{Ric}_{\overline{g}}(X,Y)
+\nabla^{2}_{\overline{g}}\varphi(X,Y)
-\rho\,S_{\overline{g}}\overline{g}(X,Y)\\
&=\textup{Ric}_{g}(X,Y)-\frac{m}{f}\nabla^{2}_{g}f(X,Y) +\nabla^{2}_{g}\eta(X,Y) + \omega\nabla^{2}_{g}f(X,Y)\\
&\quad-\rho\left(S_{g}+\frac{S_{F}}{f^2}-2m\frac{\Delta f}{f}-m(m-1) \frac{\|\nabla f\|^2}{f^2}\right)g(X,Y).
\end{split}
\end{equation}
Differentiating \eqref{EQ:Geodesic-Base} along a vertical vector field $V$ and using assumption \textup{(b)}, namely, $V(Q(X,Y))=0$, we obtain
\begin{equation}\label{EQ:Hess-type-eq}
V(\omega)\nabla^{2}_{g}f(X,Y)=\rho \frac{V(S_{F})}{f^2}g(X,Y). 
\end{equation}
Replacing $X$ with $\nabla f$ in \eqref{EQ:Hess-type-eq}, we get
\[
\frac{V(\omega)}{2}Y\left(\|\nabla f\|^2\right)=V(\omega)(\nabla^{2}_{g}f)(\nabla f,Y)=\rho \frac{V(S_{F})}{f^2}Y(f)=-\rho V(S_{F})Y\left(\frac{1}{f}\right). 
\]
This implies that 
\begin{equation}\label{EQ:omega-SF-const}
C_{V}:=\frac{V(\omega)}{2}\|\nabla f\|^2+\rho V(S_{F})\frac{1}{f}
\end{equation}
depends only on the fiber $F^{m}$. 

We next consider the vertical part of $Q$. Again by the Bishop-O'Neill formulas~\cite{bishop1969}, the vertical component of the Ricci tensor is
\begin{equation*}
\textup{Ric}_{\bar{g}}\Big{|}_{\mathcal{V}}=\textup{Ric}_{g_{F}}
- \big(f\Delta f + (m-1)\|\nabla f\|^{2}\big)g_{F}.
\end{equation*}
Moreover, for any vertical vector fields $U,V\in \mathcal{L}(F)$, the Hessian of $\varphi$ with respect to the warped metric satisfies
\[
\nabla^{2}_{\overline{g}}\varphi(U,V)
= \nabla^{2}_{g_{F}}\varphi(U,V)
+ f\nabla f(\varphi)g_{F}(U,V).
\]
Using the decomposition $\varphi=\eta+f\omega$, we arrive at
\[
\nabla^{2}_{\overline{g}}\varphi(U,V)
= f\nabla^{2}_{g_{F}}\omega(U,V)
+ f\left[\nabla f(\eta)+\omega\|\nabla f\|^2\right]g_{F}(U,V).
\]
Hence,
\begin{equation}\label{eq:Q-fiber}
\begin{split}
    Q|_{\mathcal{V}}(U,V)&=\textup{Ric}_{g_{F}}(U,V)-\left(f\Delta f+(m-1)\|\nabla f\|^2\right)g_{F}(U,V)\\
    &\quad +f\nabla_{g_{F}}^2\omega(U,V)+f\left[\nabla f(\eta)+\omega\|\nabla f\|^{2}\right]g_{F}(U,V)\\
    &\quad -\rho\left(f^{2}S_{g}+S_{F}-2mf\Delta f-m(m-1)\|\nabla f\|^2\right)g_{F}(U,V). 
\end{split}
\end{equation}

Taking the $g_{F}$-divergence of \eqref{eq:Q-fiber}, evaluating it at a vertical vector field $V$, and using assumption \textup{(c)}, namely, $\textup{div}_{g_{F}} (Q|_{\mathcal{V}})=0$, we find
\begin{equation}\label{eq:Q-fib-div}
\begin{split}
0&=(\textup{div}_{g_{F}} (Q|_{\mathcal{V}}))(V)\\
&=(\textup{div}(\textup{Ric}_{g_{F}}))(V)+f(\textup{div}(\nabla_{g_{F}}^2\omega))(V)+fV(\omega)\|\nabla f\|^{2}
    -\rho V(S_{F})\\
&=\left(\frac{1}{2}-\rho\right)V(S_{F})+f\left(\textup{Ric}_{g_{F}}(\nabla\omega,V)
    +V(\Delta\omega)\right)+fV(\omega)\|\nabla f\|^{2}\\
&=\left(\frac{1}{2}-3\rho\right)V(S_{F})+f\left(2C_{V}+\textup{Ric}_{g_{F}}(\nabla\omega,V)
    +V(\Delta\omega)\right), 
\end{split}
\end{equation}
where the third equality follows from the contracted Bianchi identity and the Bochner identity, while the fourth follows from \eqref{EQ:omega-SF-const}.

On the other hand, taking the $g_{F}$-trace of \eqref{eq:Q-fiber} restricted to the vertical space, we have 
\begin{equation}\label{eq:Q-trace}
\begin{split}
\textup{trace}_{g_{F}} (Q|_{\mathcal{V}})&=S_{F}-\left(f\Delta f+(m-1)\|\nabla f\|^2\right)m\\
&\quad +f\Delta\omega+f\left[\nabla f(\eta)+\omega\|\nabla f\|^{2}\right]m\\
&\quad -\rho\left(f^{2}S_{g}+S_{F}-2mf\Delta f-m(m-1)\|\nabla f\|^2\right)m. 
\end{split}
\end{equation}
Applying a vertical vector field $V$ to \eqref{eq:Q-trace} and using assumption \textup{(d)}, namely, $V(\textup{trace}_{g_{F}} (Q|_{\mathcal{V}}))=0$, we conclude that
\begin{equation}\label{eq:0-Q-trace}
\begin{split}
0=V(\textup{trace}_{g_{F}} (Q|_{\mathcal{V}}))&=(1-\rho m)V(S_{F})+fV(\Delta\omega)+fV(\omega)\|\nabla f\|^2m\\
&=(1-3\rho m)V(S_{F})+f\left(2mC_{V}+V(\Delta\omega)\right), 
\end{split}
\end{equation}
where the second equality follows from \eqref{EQ:omega-SF-const}. 

\smallskip
\noindent
\textbf{Case 1.} Assume $\rho\neq 1/6$. 
Then, \eqref{eq:Q-fib-div} with the assumption that $f$ is nonconstant implies that $V(S_{F})=0$, which means that $S_{F}$ is constant. 
By \eqref{EQ:Hess-type-eq}, we have $V(\omega)\nabla^{2}_{g}f=0$. 
If $V(\omega)\neq 0$, it implies $\nabla^{2}_{g}f=0$. This is impossible by~\cite[Proposition 3.6]{borges2022ricci}. 
Thus, $V(\omega)=0$, which means that $\omega$ is constant. This implies that $\varphi=\eta+f\omega$ depends only on the base $B$. 

\smallskip
\noindent
\textbf{Case 2.} Assume $\rho= 1/6$ and $m\neq 2$. 
In this case, \eqref{eq:0-Q-trace} with the assumption that $f$ is nonconstant implies that $V(S_{F})=0$. 
Then, by the same argument as in Case~1, we can prove that $\varphi$ depends only on the base $B$. 

\smallskip
\noindent
\textbf{Case 3.} Assume $\rho= 1/6$ and $m=2$. 
Define a function $\lambda$ on $B$ by $\lambda:=\frac{6}{n}f^2\Delta f$. 
Then, taking the trace of \eqref{EQ:Hess-type-eq} implies 
\begin{equation}\label{eq:parallel-om-S}
\lambda\nabla\omega=\nabla S_{F}. 
\end{equation}
\noindent
\textbf{(i)}
If $\lambda$ is nonconstant, \eqref{eq:parallel-om-S} implies $\nabla\omega=0$, and hence $\nabla S_{F}=0$. 
So, the assertion follows. \\
\textbf{(ii)}
If $\lambda=0$, we have $\nabla S_{F}=0$, and hence $V(\omega)\nabla^{2}_{g}f=0$ from \eqref{EQ:Hess-type-eq}. 
Suppose that $\omega$ is nonconstant. Then, we should have $\nabla^{2}_{g}f=0$ which means that $f$ is affine on $B$. 
However, this does not happen since there is no nonconstant affine function bounded from below on a complete Riemannian manifold. 
Thus, $\omega$ should be constant, and the assertion follows. \\
\textbf{(iii)}
Suppose now that $\lambda$ is a nonzero constant. 
At first, we remark that $\textup{Ric}_{g_{F}}=(S_{F}/2)g_{F}$ on a 2-dimensional Riemannian manifold. 
Subtracting \eqref{eq:0-Q-trace} from \eqref{eq:Q-fib-div} with the fact $\textup{Ric}_{g_{F}}=(S_{F}/2)g_{F}$ and the assumption that $f$ is nonconstant implies 
\begin{equation}\label{eq:S-omega-const}
-2C_{V}+\frac{S_{F}}{2}\nabla\omega=0. 
\end{equation}
Adding \eqref{eq:S-omega-const} multiplied by $2f$ to \eqref{EQ:omega-SF-const} multiplied by $4f$ implies 
\[
\left(fS_{F}-2f\|\nabla f\|^2-\frac{2}{3}\lambda\right)\nabla\omega=0, 
\]
where we also used \eqref{eq:parallel-om-S}. 
Suppose that $\omega$ is nonconstant. 
Then, there is a point $q_{0}\in F$ and an open neighborhood $\Omega\subset F$ of $q_{0}$ such that $\nabla\omega\neq 0$ on $\Omega$. 
Then, on $B\times \Omega$, it should hold that $fS_{F}-2f\|\nabla f\|^2-2\lambda/3=0$. 
Acting the gradient with respect to $g_{F}$ implies $f\nabla S_{F}=0$. Since $f>0$, we have $\nabla S_{F}=0$ and it implies that $S_{F}$ is constant. Then, by the same argument as in Case~1, we can prove that $\varphi$ depends only on the base $B$. 
\end{proof}

Now, applying Lemma \ref{New-lem-1} to the Ricci-Bourguignon flow, we prove Proposition~\ref{Proposition-1}.

\begin{proof}[\bf Proof of Proposition~\ref{Proposition-1}]
We only need to work at an arbitrary fixed time~$t$. 
Comparing \eqref{tens-Q} and \eqref{RBF-2}, we see 
\[Q=-\frac{1}{2}\frac{\partial}{\partial t}\overline{g}(t)=-\frac{1}{2}\frac{\partial}{\partial t}g(t)-f(t)\frac{\partial f}{\partial t}(t)g_{F}.\]
It is straightforward to verify that $Q$ satisfies assumptions (a)--(d) of Lemma~\ref{New-lem-1}. The conclusion then follows.
\end{proof}

We are ready to prove the main result of this section.
\begin{theorem}\label{thm1} 
Let $(B^{n},g(t))$ be a smooth one-parameter family of complete Riemannian manifolds, $(F^{m},g_{F})$ a complete Riemannian manifold, and $f(t)$ a smooth family of positive nonconstant functions on $B^n$. We define
\begin{equation*}
\overline{g}(t)=g(t)+f(t)^2g_{F}.
\end{equation*}
Then $\overline{g}(t)$ evolves by the Ricci-Bourguignon flow \eqref{Eq-RBFD} if and only if $(F^{m},g_{F})$ is an Einstein manifold with constant scalar curvature $S_F$, and the flow reduces to the following coupled system on the base:
\begin{equation}\label{thm1-01}
\begin{split}
\frac{\partial}{\partial t}g(t)=&-2\textup{Ric}_{g(t)}+\frac{2m}{f}\nabla^{2}_{g(t)}f-2\nabla^{2}_{g(t)}\varphi\\
&+2\rho\left(
S_{g(t)}
+ \frac{S_{F}}{f^{2}}
- 2m\frac{\Delta f}{f}
- m(m-1)\frac{\|\nabla f\|^{2}}{f^{2}}
\right) g(t),
\end{split}
\end{equation}
and
\begin{equation}\label{thm1-02}
\begin{split}
\frac{\partial f}{\partial t}\!=\!(1\!-\!2m\rho)\Delta f+(1\!-\!m\rho)(m\!-\!1)\frac{\|\nabla f\|^{2}}{f}\!-\!\nabla f(\varphi)+\rho S_{g(t)}f+\frac{m\rho\!-\!1}{m}\frac{S_{F}}{f}.
\end{split}
\end{equation}
In particular, assuming $\rho\notin\{\frac{1}{m+1},\frac{1}{2m}\}$, we may rewrite \eqref{thm1-02} in the equivalent form 
\begin{equation*}
\dfrac{\partial u}{\partial t}=\Delta_{\varphi} u-2m\rho\Delta u+\frac{\rho}{\sigma} S_{g(t)}u+\dfrac{m\rho-1}{m\sigma}S_{F}u^{1-2\sigma},
\end{equation*}
where 
\begin{equation*}
u=f^{\frac{1}{\sigma}},\qquad  \sigma=\frac{1-2m\rho}{m-\rho m^{2}-m\rho },\qquad\mbox{and}\qquad \Delta_{\varphi}u:= \Delta u-\langle\nabla \varphi,\nabla u\rangle.
\end{equation*}
\end{theorem}

\begin{proof} Differentiating the warped metric $\overline{g}(t)=g(t)+f(t)^{2}g_{F}$
with respect to $t$ yields
\begin{equation*}
    \frac{\partial}{\partial t}\bar{g}(t)\ \Big{|}_{\mathcal{H}}=\frac{\partial}{\partial t}g(t),\qquad \frac{\partial}{\partial t}\bar{g}(t)\ \Big{|}_{\mathcal{V}}= 2f(t)\frac{\partial f}{\partial t} g_{F}.
\end{equation*}
The warped product structure also determines the corresponding decomposition of the curvature terms. From~\cite{bishop1969}, one has
\begin{equation*}
    \begin{split}
        \textup{Ric}_{\bar{g}(t)}\Big{|}_{\mathcal{H}}
= \textup{Ric}_{g(t)}-\frac{m}{f}\nabla^{2}_{g(t)}f,\quad
\textup{Ric}_{\bar{g}(t)}\Big{|}_{\mathcal{V}}=\textup{Ric}_{g_{F}}
- \big(f\Delta f + (m-1)\|\nabla f\|^{2}\big)g_{F}.
    \end{split}
\end{equation*}
Since $f(t)$ is a family of nonconstant warping functions, Proposition~\ref{Proposition-1} shows that $\varphi(\cdot,t)$ depends only on $B^{n}$. Consequently,
\begin{equation*}
    \begin{split}
        \nabla^{2}_{\bar{g}(t)}\varphi\Big{|}_{\mathcal{H}}=\nabla^{2}_{g(t)}\varphi,\qquad
        \nabla^{2}_{\bar{g}(t)}\varphi\Big{|}_{\mathcal{V}}=f\nabla f (\varphi) g_{F}.
    \end{split}
\end{equation*}
Moreover, the scalar curvature formula for warped products gives
\begin{equation*}
    \rho S_{\bar{g}(t)}\bar{g}(t)
= \rho\left(
S_{g(t)}
+ \frac{S_{F}}{f^{2}}
- 2m\frac{\Delta f}{f}
- m(m-1)\frac{\|\nabla f\|^{2}}{f^{2}}
\right) \bar{g}(t).
\end{equation*}
By restricting the Ricci-Bourguignon flow to the vertical distribution, we arrive at $\textup{Ric}_{g_{F}}=\mu g_{F}$, where
\begin{equation*}
\mu=\rho S_{F}-f\frac{\partial f}{\partial t}+(1-2m\rho)f\Delta f+(1-\rho m)(m-1)\|\nabla f\|^{2}-f\nabla f (\varphi)+\rho f^{2}S_{g(t)}.
\end{equation*}
Thus $(F^{m},g_F)$ is an Einstein manifold, and since $\mu=S_{F}/m$, we obtain~\eqref{thm1-02}.

On the other hand, restricting the Ricci-Bourguignon flow to the horizontal distribution, we obtain~\eqref{thm1-01}.

The converse is obtained by substituting equations \eqref{thm1-01} and \eqref{thm1-02} into the Ricci-Bourguignon flow equation.
\end{proof}

\section{Constructing warped metric solutions to the Ricci-Bourguignon flow} \label{CWMSRBF}

In this section, we combine Theorem~\ref{thm1} with an \textit{ansatz} method to construct explicit solutions to the Ricci-Bourguignon flow by reducing a PDE to an ODE. One of the main tools for identifying useful ansatz is the theory of Lie symmetry groups for PDEs (see Olver's book \cite{olver2000applications} or Bluman and Kumei's book \cite{bluman1989symmetries}). A classical example of this method is the Bryant soliton, which arises from a rotationally symmetric ansatz and yields a complete gradient steady Ricci soliton on $\mathbb{R}^{n}$, $n\geqslant3$, unique up to homothety (see Chow et al.~\cite{chow2007ricci}). Also, by using a rotationally symmetric ansatz, T.~Ivey~\cite{ivey1994new} obtained a one-parameter family of complete noncompact gradient steady Ricci solitons of doubly warped product type on $\mathbb{R}^{k+1}\times N$, where $N$ is a compact Einstein manifold with positive scalar curvature. In the same spirit, S.~Angenent and D.~Knopf~\cite{angenent2022ricci} constructed complete gradient shrinking Ricci solitons within a doubly warped product setting.

Here, we begin by looking for solutions $\bar{g}(x,t)=g(x,t)+f(x,t)^{2}g_{F}$ to the Ricci-Bourguignon flow~\eqref{Eq-RBFD}  with time-independent $\varphi$, Ricci-flat metric $g_{F}$ and
\[g(x,t)=a(t)g_{0}(x),\quad\qquad f(x,t)=b(t)f(x),\]
where $a(t)=1+c_{0}t>0$ and $b(t)=a(t)^{\frac{c_{1}}{c_{2}}}$ for some $c_{0}, c_{1}, c_{2}\in\mathbb{R}$, $c_{2}\neq0$. Thus, the time dependence is completely determined by a homothetic deformation of the base metric and the corresponding scaling of the warping function. It follows from Theorem~\ref{thm1} that the coupled flow equations \eqref{thm1-01}--\eqref{thm1-02} are then equivalent to the following system for the initial metric
\begin{equation}\label{2a}
\begin{split}
\textup{Ric}_{g_{0}}\!-\frac{m}{f}\nabla_{g_{0}}^{2}f+\nabla_{g_{0}}^{2}\varphi+\frac{c_{0}}{2} g_{0}=\rho\!\left(\!S_{g_{0}}\!
- 2m\frac{\Delta_{g_{0}} f}{f}
- m(m-1)\frac{\|\nabla f\|_{g_{0}}^{2}}{f^{2}}\!\right)\!g_{0}, 
\end{split}
\end{equation}
and
\begin{equation}\label{2b}
\begin{split}
\left(\frac{c_{0}c_{1}}{c_{2}}-\rho S_{g_0}\right)\!f^{2}
= (1-2m\rho)f\Delta_{g_0}f +(1-m\rho)(m-1)\|\nabla f\|_{g_0}^2 -f \nabla f(\varphi).
\end{split}
\end{equation}

Now, we introduce an ansatz on the base of the warped product. More precisely, consider the conformal metric $g_{0}=\mu^{-2}\langle\ ,\ \rangle$ on $\mathbb{R}^n$, where $\langle\ ,\ \rangle$ denotes the flat metric on $\mathbb{R}^{n}$. For a fixed unit vector $\alpha=\left(\alpha_1, \ldots, \alpha_n\right)\in\mathbb{R}^{n}$, we introduce the ansatz $\xi(x)=\langle x,\alpha\rangle$. Then, we look for smooth functions $\mu, f, \varphi:I\subseteq \mathbb{R} \rightarrow \mathbb{R}$ with $\mu>0$ and $f>0$ such that $\mu \circ \xi$, $f \circ \xi$ and $\varphi \circ \xi$
defined on $\Omega^{n}:=\xi^{-1}(I)$ satisfy~\eqref{2a} and~\eqref{2b}. For simplicity, we shall denote both the functions on $I$ and their pullbacks along $\xi$ by the same symbols. 

For the conformal metric $g_{0}=\mu^{-2} \langle\ ,\ \rangle$, the Ricci tensor and the scalar curvature are given by
\begin{equation}\label{ricci_base}
\begin{split}
(\textup{Ric}_{g_{0}})_{ij}&=\frac{1}{\mu^2}\left\{(n-2)\alpha_{i}\alpha_{j} \mu\mu^{\prime \prime}+\left[\mu\mu^{\prime \prime} -(n-1)\left(\mu^{\prime}\right)^2\right] \delta_{ij}\right\}\\
S_{g_{0}}&=\sum_{k=1}^n \mu^2\left(\textup{Ric}_{g_{0}}\right)_{k k}=(n-1)\left[2 \mu \mu^{\prime \prime}-n\left(\mu^{\prime}\right)^2\right].
\end{split}
\end{equation}
Moreover, for a smooth function $f$, the Hessian and the Laplacian with respect to  $g_{0}$ take the form
\begin{equation}\label{hes}
\begin{aligned}
(\nabla^{2}_{g_{0}}f)_{i j}&=\alpha_i \alpha_j f^{\prime \prime}+\left(2 \alpha_i \alpha_j-\delta_{i j} \right)\mu^{-1} \mu^{\prime} f^{\prime}\\[4pt]
\Delta_{g_{0}}f&=\sum_k \mu^2 (\nabla_{g_{0}}^{2}f)_{k k}=\mu^2\left[f^{\prime \prime}-(n-2) \mu^{-1} \mu^{\prime} f^{\prime}\right],
\end{aligned}
\end{equation}
while the remaining terms are described by
\begin{equation}\label{finais}
\begin{aligned}
\nabla f(\varphi)= \mu^2 f^{\prime} \varphi^{\prime},
\qquad\|\nabla f\|_{g_{0}}^2= \mu^2\left(f^{\prime}\right)^2.
\end{aligned}
\end{equation}

Hence, substituting \eqref{ricci_base}--\eqref{finais} into \eqref{2a}--\eqref{2b}, we arrive at the following characterization. 

\begin{proposition}\label{proposition-invariance}
Let $\xi$ be as above, and suppose that the functions $\mu$, $f$, and $\varphi$ depend only on $\xi$, namely, $\mu=\mu\circ\xi$, $f=f\circ\xi$, $\varphi=\varphi\circ\xi$. Let $(F^{m},g_{F})$ be a Ricci-flat manifold, and consider on $\Omega^{n}\times F^{m}$ the family of metrics
\[
\bar{g}(x,t)=(1+c_{0}t)\mu(x)^{-2}\langle\ ,\ \rangle+(1+c_{0}t)^{\frac{2c_{1}}{c_{2}}}f(x)^{2}g_{F},\quad c_{0},c_{1}\in\mathbb{R},\quad c_{2}\neq0.
\]
Then $\bar{g}(x,t)$ evolves by the Ricci-Bourguignon flow~\eqref{Eq-RBFD} with time-independent $\varphi$ if and only if $\mu$, $f$ and $\varphi$ satisfy
     
\begin{equation}\label{Prop2-1}
 (n-2) \frac{\mu^{\prime \prime}}{\mu}- m\frac{f^{\prime \prime}}{f} -2 m\frac{f^{\prime}}{f} \frac{\mu^{\prime}}{\mu} + \varphi^{\prime\prime}+2\frac{\mu^{\prime}}{\mu}\varphi^{\prime}=0,
\end{equation}

\begin{equation}\label{Prop2-2}
    \begin{aligned}
\left(1-2(n-1)\rho\right)\frac{\mu^{\prime \prime}}{\mu}-&(n-1)(1-n\rho)\left(\frac{\mu^{\prime}}{\mu}\right)^2+m\left(1-2(n-2)\rho\right) \frac{\mu^{\prime}}{\mu}\frac{f^{\prime}}{f}\\
&\qquad+2m\rho\frac{f''}{f}+m(m-1)\rho\left(\frac{f^{\prime}}{f}\right)^2-\frac{\mu^{\prime}}{\mu} \varphi^{\prime}+\frac{c_{0}}{2\mu^2} =0
\end{aligned}
\end{equation}
and
\begin{equation}\label{Prop2-3}
\begin{aligned}
-(1-2m\rho)\frac{f''}{f}
&+(n-2)(1-2m\rho)\frac{\mu'}{\mu}\frac{f'}{f}
-(m-1)(1-m\rho)\left(\frac{f'}{f}\right)^2
+\varphi'\frac{f'}{f} \\
&-2\rho(n-1)\frac{\mu''}{\mu}
+\rho n(n-1)\left(\frac{\mu'}{\mu}\right)^2
+\frac{c_0c_{1}}{c_{2}\mu^2}
=0.
\end{aligned}
\end{equation}
\end{proposition}

The following three examples illustrate how the previous proposition can be used to construct explicit solutions to the Ricci-Bourguignon flow.

\begin{example}\label{Exa*} Let us consider the hyperbolic space $\mathbb{H}^{n}$, namely, the open half space $\mathbb{R}_{+}^{n}=
\{
(x_{1},  \dots, x_{n})\in\mathbb{R}^{n}: x_{n}>0
\}$ with the metric $g_{\mathbb{H}}=x_{n}^{-2}\langle\ ,\ \rangle$ and $(F^{m},g_{F})$ a complete Ricci-flat manifold. Besides, consider $\alpha_{1}=\cdots=\alpha_{n-1}=0$ and $\alpha_{n}=1$. Then the functions
\begin{equation*}
    \mu(x_{n})=x_{n},\qquad f(x_{n})=x_{n},\qquad \varphi(x_{n})=2m\ln (x_{n}).
\end{equation*}
satisfy \eqref{Prop2-1}, \eqref{Prop2-2} and \eqref{Prop2-3} provided that
\begin{equation*}
    c_{0}=2\left[(n+m-1)-\rho\left(n^{2}-n-2nm+m^2+3m\right)\right],
\end{equation*}
and
\begin{equation*}
    \dfrac{c_{0}c_{1}}{c_{2}}=-\left[(n+m-1)+\rho\left(n^{2}-n-2nm+m^2+3m\right)\right].
\end{equation*}
Therefore, the family of metrics
\[
\bar{g}(x,t)=(1+c_{0}t)x_{n}^{-2}\langle\ ,\ \rangle+(1+c_{0}t)^{\frac{2c_{1}}{c_{2}}}x_{n}^{2}g_{F}
\]
defines a family of complete metrics on $\mathbb{H}^{n}\times F^{m}$ evolving by the Ricci-Bourguignon flow~\eqref{Eq-RBFD}. In particular, $\bar{g}(x,t)$ is a purely homothetic self-similar solution if $2c_{1}=c_{2}$, and non-homothetic otherwise.

 For instance, if we take $n=2$, $m=1$, $\rho=1/3$, $c_{0}=8/3$, $c_{1}=-1$, $c_{2}=1$ and 
\begin{equation*}
    \mu(x_{2})=x_{2},\qquad f(x_{2})=x_{2},\qquad \varphi(x_{2})=2\ln(x_{2}),
\end{equation*}
then we obtain that 
\[
\bar{g}(x,t)=\left(1+\frac{8}{3}t\right)x_{2}^{-2}\langle\ ,\ \rangle+\left(1+\frac{8}{3}t\right)^{-2}x_{2}^{2}g_{F},\quad \mbox{for}\ \ t>-\frac{3}{8}
\]
is an immortal solution to the Ricci-Bourguignon flow~\eqref{Eq-RBFD} on $\mathbb{H}^{2}\times F^{1}$. 

On the other hand, if we take $n=2$, $m=1$, $\rho=2$, $c_{0}=-4$, $c_{1}=3$, $c_{2}=2$, then we obtain that
\[
\bar{g}(x,t)=\left(1-4t\right)x_{2}^{-2}\langle\ ,\ \rangle+\left(1-4t\right)^{3}x_{2}^{2}g_{F},\quad \mbox{for}\ \ t<\frac{1}{4}
\]
is an ancient solution to the Ricci-Bourguignon flow~\eqref{Eq-RBFD} on $\mathbb{H}^{2}\times F^{1}$. This solution is a non-homothetic self-similar solution to~\eqref{Eq-RBFD}. 
\end{example}

\begin{example}
    Let $\mathbb{R}_{\ast}^{n}=\{x=(x_{1},\dots, x_{n})\in\mathbb{R}^{n}:\xi(x)=\alpha_{1}x_{1}+\dots+\alpha_{n}x_{n}>0\}$ equipped with the metric $\mu^{-2}\langle\ ,\ \rangle$ and $(F^{m}, g_{F})$ a complete Ricci-flat manifold. Let us consider 
\begin{equation*}
    \mu(\xi)=\xi,\quad f(\xi)=\frac{L}{\xi},\quad \varphi(\xi)\equiv\varphi_{0}\in\mathbb{R}, \quad \rho=0, \quad c_{0}=2(m+n-1),\quad \frac{c_{1}}{c_{2}}=\frac{1}{2}.
\end{equation*}
Then, Proposition~\ref{proposition-invariance} shows that
\begin{equation*}
\bar{g}(x,t)=\left(1+c_{0}t\right)\xi(x)^{-2}\langle\ ,\ \rangle+\left(1+c_{0}t\right)L^{2}\xi(x)^{-2}g_{F}, \quad \mbox{for}\quad 1+c_{0}t>0
\end{equation*}
defines a family of complete metrics on $\mathbb{R}_{\ast}^{n}\times F^{m}$ evolving by the Ricci flow. To prove that $\bar{g}(x,t)$ is complete, it suffices to show that $g=\xi(x)^{-2}\langle\ ,\ \rangle$ is complete on $\mathbb{R}_{\ast}^{n}$. Since $\alpha\neq 0$, there exists an orthogonal map $A\in O(n)$ such that
\begin{equation*}
    \xi(x)=\|\alpha\|y_{n},\qquad y=Ax.
\end{equation*}
Since $A$ is orthogonal, the Euclidean metric is preserved, that is,
\begin{equation*}
\langle\ ,\ \rangle = dy_1^2+\cdots+dy_n^2,
\end{equation*}
and the domain $\mathbb{R}_*^n$ is transformed into $\{y\in\mathbb{R}^n:y_n>0\}$. Hence, $g=\|\alpha\|^{-2}g_{\mathbb{H}}$ which is complete.
\end{example}

\begin{example}
 Let $\mathbb{R}_{\ast}^{2}=\{x=(x_{1},x_{2})\in\mathbb{R}^{2}:\xi(x)=\alpha_{1}x_{1}+\alpha_{2}x_{2}>0\}$ equipped with the metric $\mu^{-2}\langle\ ,\ \rangle$ and $(F^{1}, g_{F})$ a complete one-dimensional manifold. Let us consider 
\begin{equation*}
    \mu(\xi)=e^{\xi},\quad f(\xi)=1+\frac{e^{-2\xi}}{2},\quad \varphi(\xi)=-\frac{e^{-2\xi}}{4}, \quad \rho=\frac{1}{4},\quad c_{0}=1\quad\mbox{and}\quad \frac{c_{1}}{c_{2}}=1.
\end{equation*}
Then, from  Proposition~\ref{proposition-invariance}, we have that 
\[
\bar{g}(x,t)=\left(1+t\right)e^{-2\xi(x)}\langle\ ,\ \rangle+\left(1+t\right)^{2}\left(1+\frac{e^{-2\xi(x)}}{2}\right)^{2}g_{F},\quad t>-1
\]
defines an immortal solution to the Ricci-Bourguignon flow~\eqref{Eq-RBFD} on $\mathbb{R}_{\ast}^{2}\times F^{1}$. In this case, the family $\bar{g}(x,t)$ is not complete since the metric on the base manifold is not complete.
\end{example}

The next example produces a metric that is a scaling of an Einstein metric, and no ansatz is required.

\begin{example}\label{Ex2}
Consider the real line $\mathbb{R}$ endowed with the one-parameter family of metrics
\begin{equation*}
g(t)=\bigl(1+2(n-1)t\bigr)ds^{2}, 
\qquad 
t\in\Bigl(-\tfrac{1}{2(n-1)},\infty\Bigr),
\end{equation*}
and let $\left(\mathbb{R}^{n-1}, g_{\mathrm{E}}\right)$ be the Euclidean fiber. Set
\begin{equation*}
a(t)=1+2(n-1)t
\quad \mbox{and} \quad
f(s,t)=\sqrt{a(t)}e^{s}.
\end{equation*}
We then consider the metric
\begin{equation*}
\bar g(t)=a(t)ds^{2}+f(s,t)^{2}g_{E}.
\end{equation*}
This metric can be written as $\overline{g}(t)=a(t)h$, where $h=ds^{2}+e^{2s}g_{E}$. The metric $h$ is Einstein and satisfies $\mathrm{Ric}_h=-(n-1)h$. Since $a(t)$ depends only on time, we have $\textup{Ric}_{\bar g(t)}=\textup{Ric}_h$. Moreover,
\begin{equation*}
\frac{\partial}{\partial t}\bar g(t)
=a'(t)h
=2(n-1)h
=-2\textup{Ric}_h
=-2\textup{Ric}_{\bar g(t)},
\end{equation*}
which shows that $\bar g(t)$ is an immortal solution to the Ricci flow. The warping function satisfies $\ln f(s,t)=s+\tfrac{1}{2}\ln a(t)$. Consequently, the growth condition~\eqref{sharp-condition} is sharp with respect to the spatial factor.
\end{example}

\section{Proof of the main results}
\begin{proof}[\bf Proof of Theorem~\ref{main}] 
In the proof of Theorem~\ref{main}, all differential operators and norms are taken with respect to the evolving metric $g(t)$. More precisely, $\Delta$, $\nabla$,
$\langle\cdot,\cdot\rangle$ and $\|\cdot\|$ denote the Laplacian,
the gradient, the metric inner product, and the
norm induced by $g(t)$, all depending on the time parameter $t$. 

We begin by fixing $p\in(0,\infty)$ and setting $h=p\ln u$. In this case, the parabolic equation~\eqref{Eq-Parabolic*} becomes
\begin{equation*}
\frac{\partial h}{\partial t}=\Delta_{\varphi} h-a\Delta h+\frac{1-a}{p}\|\nabla h\|^{2}+pb
+pce^{\frac{h}{p}(\alpha-1)}.\end{equation*}
Suppose that there exists a constant $q$ such that $q-h\geqslant \delta > 0$ for some $\delta$. Then, for all $(x,t)\in \mathcal{Q}_{R,T}$ the function
\begin{equation*}
G:=\frac{\|\nabla h\|^{2}}{(q-h)^{2}}
\end{equation*}
satisfies
\begin{equation}\label{Eq:06}
\begin{split}
\frac{1}{2}\left(\bar{a}\Delta_{\overline{\varphi}}-\frac{\partial}{\partial t}\right)G\geqslant&-\left(k_{1}+k_{2}\right)G+\frac{\bar{a}(p-q+h)}{p(q-h)}\langle\nabla h,\nabla G\rangle+\frac{\bar{a}(q-h)}{p}G^{2}\\[1pt]
&-\frac{p\langle\nabla h,\nabla b\rangle}{(q-h)^{2}}-\frac{pb}{q-h}G-\left(\alpha-1+\frac{p}{q-h}\right)c e^{\frac{h}{p}(\alpha-1)}G.
\end{split}
\end{equation}
where we use $\bar{a}=1-a$. In fact, first note that applying the Leibniz rule and using that $\overline{\varphi}=\frac{\varphi}{\bar{a}}$, we obtain
\begin{equation*}
\begin{split}
\frac{\partial}{\partial t}\|\nabla h\|^{2}=-&\left(\frac{\partial }{\partial t}g(t)\right)(\nabla h,\nabla h)+2\,\Big{\langle}\nabla h,\nabla \frac{\partial h}{\partial t}\Big{\rangle}\\
=-&\left(\frac{\partial }{\partial t}g(t)\right)(\nabla h,\nabla h)+2\bar{a}\langle\nabla h, \nabla\Delta_{\overline{\varphi}}h\rangle+\frac{2\bar{a}}{p}\langle\nabla h,\nabla \|\nabla h\|^{2}\rangle\\[1pt]
+&\,2p\langle\nabla h,\nabla b\rangle+2c(\alpha-1)(e^{\frac{h}{p}})^{\alpha-1}\|\nabla h\|^{2}.
\end{split}
\end{equation*}
Consequently,
\begin{equation}\label{Eq:Partial-G}
\begin{split}
\frac{\partial G}{\partial t}=&\dfrac{\frac{\partial}{\partial t}\|\nabla h\|^{2}}{(q-h)^{2}}+2\frac{\partial h}{\partial t}\dfrac{\|\nabla h\|^{2}}{(q-h)^{3}}\\[1pt]
=&-\dfrac{\left(\frac{\partial }{\partial t}g(t)\right)(\nabla h,\nabla h)}{(q-h)^{2}}+\dfrac{2\bar{a}\langle\nabla h, \nabla\Delta_{\overline{\varphi}}h\rangle}{(q-h)^{2}}+\frac{2\bar{a}}{p}\dfrac{\langle\nabla h,\nabla \|\nabla h\|^{2}\rangle}{(q-h)^{2}}\\[1pt]
&+\dfrac{2p\langle\nabla h,\nabla b\rangle}{\,(q-h)^{2}}+\dfrac{2\bar{a}\Delta_{\overline{\varphi}}h\|\nabla h\|^{2}}{(q-h)^{3}}+\frac{2\bar{a}}{p}\dfrac{\|\nabla h\|^{4}}{(q-h)^{3}}+\dfrac{2pb\|\nabla h\|^{2}}{(q-h)^{3}}\\[1pt]
&+2c\left(\alpha-1+\frac{p}{q-h}\right)(e^{\frac{h}{p}})^{\alpha-1}G.
\end{split}
\end{equation}
Now, from the definition of $G$, we compute
\begin{align*}
\Delta G&=\frac{\Delta\|\nabla h\|^2}{(q-h)^2}+\frac{4 \langle\nabla h, \nabla\| \nabla h\|^2\rangle}{(q-h)^3}+\frac{2\Delta h\|\nabla h\|^2 }{(q-h)^3}+\frac{6\|\nabla h\|^4}{(q-h)^4},\\[1pt]
\langle\nabla G, \nabla \overline{\varphi}\rangle&=\frac{\langle\nabla\|\nabla h\|^2, \nabla \overline{\varphi}\rangle}{(q-h)^2}+\frac{2\langle\nabla h, \nabla \overline{\varphi}\rangle\|\nabla h\|^2}{(q-h)^3}.
\end{align*}
Therefore,
\begin{equation}\label{Eq:Delta-varphi-G}
\begin{split}
\Delta_{\overline{\varphi}} G=&\frac{\Delta_{\overline{\varphi}}\|\nabla h\|^2}{(q-h)^2}+\frac{4\langle\nabla h, \nabla\| \nabla h\|^2\rangle}{(q-h)^3}+\frac{2 \Delta_{\overline{\varphi}} h\|\nabla h\|^2}{(q-h)^3}+\frac{6\|\nabla h\|^4}{(q-h)^4}\\[1pt]
=&\frac{2\|\nabla^{2}h\|^{2}}{(q-h)^{2}}+\frac{2\langle\nabla h, \nabla(\Delta_{\overline{\varphi}}h)\rangle}{(q-h)^2}+\frac{2\textup{Ric}^{\overline{\varphi}}(\nabla h,\nabla h)}{(q-h)^{2}}+\frac{4\langle\nabla h,\nabla\|\nabla h\|^{2}\rangle}{(q-h)^{3}}\\[1pt]
&+\frac{2\Delta_{\overline{\varphi}}h\|\nabla h\|^{2}}{(q-h)^{3}}+\frac{6\|\nabla h\|^{4}}{(q-h)^{4}},
\end{split}
\end{equation}
where in the last equality we have used the Bochner formula
\begin{equation*}
\frac{1}{2} \Delta_{\overline{\varphi}}\|\nabla h\|^2=\|\nabla^2 h\|^2+\langle\nabla h, \nabla(\Delta_{\overline{\varphi}} h)\rangle+\textup{Ric}^{\overline{\varphi}}(\nabla h, \nabla h).
\end{equation*}
Combining \eqref{Eq:Partial-G} and \eqref{Eq:Delta-varphi-G}, we get
\begin{equation}\label{Eq:Delta-partial-G}
\begin{split}
\left(\bar{a}\Delta_{\overline{\varphi}}-\frac{\partial}{\partial t}\right)G=\ &2\bar{a}\Bigg{\|}\frac{\nabla^{2}h}{q-h}+\frac{dh\otimes dh}{(q-h)^{2}}\Bigg{\|}^{2}+\frac{2\bar{a}\langle\nabla h,\nabla G\rangle}{q-h}-\frac{2\bar{a}}{p}\langle\nabla h,\nabla G\rangle\\[1pt]
&+\frac{\left(\frac{\partial}{\partial t}g(t)+2\bar{a}\textup{Ric}^{\overline{\varphi}}_{g(t)}\right)(\nabla h,\nabla h)}{(q-h)^{2}}-\frac{2p\langle\nabla h,\nabla b\rangle}{(q-h)^{2}}\\[1pt]
&-\frac{2pb\|\nabla h\|^{2}}{(q-h)^{3}}-2c\left(\alpha-1+\frac{p}{q-h}\right)(e^{\frac{h}{p}})^{\alpha-1}G\\[1pt]
&+\frac{2\bar{a}}{p}\frac{\|\nabla h\|^{4}}{(q-h)^{3}},
\end{split}
\end{equation}
where we have used 
\begin{equation*}
\begin{split}
2\bar{a}\,\Bigg{\|}\frac{\nabla^{2}h}{q-h}+\frac{dh\otimes dh}{(q-h)^{2}}\Bigg{\|}^{2}=&\,\frac{2\bar{a}\|\nabla^{2}h\|^{2}}{(q-h)^{2}}+\frac{2\bar{a}\langle\nabla h,\nabla\|\nabla h\|^{2}\rangle}{(q-h)^{3}}+\frac{2\bar{a}\|\nabla h\|^{4}}{(q-h)^{4}},\\[1pt]
\frac{2\bar{a}\langle\nabla h,\nabla G\rangle}{q-h}=&\,\frac{2\bar{a}\langle\nabla h,\nabla\|\nabla h\|^{2}\rangle}{(q-h)^{3}}+\frac{4\bar{a}\|\nabla h\|^{4}}{(q-h)^{4}},\\[1pt]
\frac{2\bar{a}\langle\nabla h,\nabla G\rangle}{p}=&\,\frac{2\bar{a}}{p}\frac{\langle\nabla h,\nabla\|\nabla h\|^{2}\rangle}{(q-h)^{2}}+\frac{4\bar{a}}{p}\frac{\|\nabla h\|^{4}}{(q-h)^{3}}.
\end{split}
\end{equation*}
Along the family of metrics $g(t)$, one has
\begin{equation}\label{Eq:Estimates}
\begin{split}
\frac{\partial}{\partial t}g(t)+2\bar{a}\textup{Ric}^{\overline{\varphi}}_{g(t)}&=\frac{\partial}{\partial t}g(t)+2\bar{a}\textup{Ric}_{g(t)}+2\nabla^{2}\varphi\geqslant -2k_{2}g(t)-2k_{1}g(t).
\end{split}
\end{equation}
Combining \eqref{Eq:Delta-partial-G} with \eqref{Eq:Estimates}, we arrive at \eqref{Eq:06}.

The desired parabolic gradient estimate is obtained by means of a localization argument combined with a cut-off function and an application of the maximum principle to a suitably defined function. Let $r(x,t)=\mathrm{dist}_{t}(x,x_{0})$, where $x_{0}\in M$ is fixed. We then introduce the following function
\begin{equation}\label{thm-2:eq4}\phi:M\times \left[t_{0}-T, t_{0}\right]\longrightarrow\mathbb{R},\quad \phi(x,t)=\overline{\phi}(r(x,t),t),
\end{equation}
where $\overline{\phi}$ is a smooth cut-off function given by the following lemma:
\begin{lemma}[\cite{bailesteanu2010gradient}]Fix $t_{0}\in\mathbb{R}$ and let $R, T>0$. For any $\tau\in(t_{0}-T,t_{0}]$ there exists a smooth function $\bar{\phi}:[0, \infty) \times\left[t_{0}-T, t_{0}\right] \rightarrow \mathbb{R}$ satisfying 
\begin{enumerate}[(a)]
\item [\textup{(a)}] $0 \leqslant \bar{\phi}(r, t) \leqslant 1$ in $[0, R] \times\left[t_{0}-T, t_{0}\right]$, and $\textup{supp}(\bar{\phi})\subset [0, R] \times\left[t_{0}-T, t_{0}\right]$.\vspace{0.1cm}
\smallskip
\item [\textup{(b)}] $\bar{\phi}(r, t)=1$ in $[0, \frac{R}{2}] \times\left[\tau, t_0\right]$ and $\frac{\partial}{\partial r} \bar{\phi}(r, t)=0$ in $[0, \frac{R}{2}] \times\left[t_{0}-T, t_{0}\right]$.\vspace{0.1cm}
\smallskip
\item [\textup{(c)}] $\bar{\phi}\left(r, t_{0}-T\right)=0$ for all $r \in[0, \infty)$ and there exists a constant $c$ such that
$$\left|\frac{\partial \bar{\phi}}{\partial t}\right| \leqslant \dfrac{C}{\tau-t_{0}+T} \bar{\phi}^{\frac{1}{2}},\qquad\mbox{in}\quad  [0, \infty) \times\left[t_{0}-T, t_{0}\right].$$ 
\smallskip
\item [\textup{(d)}] For any $\epsilon \in(0,1)$ there exist $C_\epsilon>0$ such that
$$-\dfrac{C_\epsilon}{R} \bar{\phi}^\epsilon \leqslant \frac{\partial \bar{\phi}}{\partial r} \leqslant 0,\qquad \left|\frac{\partial^{2}\bar{\phi}}{\partial r^{2}} \right| \leqslant \dfrac{C_\epsilon}{R^2} \bar{\phi}^\epsilon\qquad\mbox{in}\quad[0, \infty) \times\left[t_{0}-T, t_{0}\right].$$ 
\end{enumerate}
\end{lemma}

Now, starting with the localized function $\phi G$, it is clear that this choice puts us in the right setting to invoke the maximum principle; in particular, by applying the method to $\phi G$, we obtain the following differential inequality:
\begin{equation}\label{estimate1}
\begin{split}
\dfrac{1}{2}\left(\bar{a}\Delta_{\overline{\varphi}}-\frac{\partial}{\partial t}\right)&(\phi G)-
\bar{a}\left[\dfrac{(p-q+h)}{p(q-h)}\nabla h+\dfrac{\nabla\phi}{\phi}\right]
\nabla(\phi G)=\\ 
&=\dfrac{1}{2}\bar{a}\phi\Delta_{\overline{\varphi}}G+\bar{a}\langle\nabla\phi,\nabla G\rangle+\dfrac{1}{2}\bar{a}G\Delta_{\overline{\varphi}}\phi-
\frac{1}{2}\left(G\frac{\partial\phi}{\partial t}+\phi\frac{\partial G}{\partial t}\right)\\ 
&\quad-\bar{a}
\left[\dfrac{(p-q+h)}{p(q-h)}\nabla h+\dfrac{\nabla\phi}{\phi}\right](\phi\nabla G+G\nabla\phi)\\ 
&=
\dfrac{1}{2}\bar{a}\phi\Delta_{\overline{\varphi}}G-\frac{1}{2}\phi\frac{\partial G}{\partial t}+
\dfrac{1}{2}\bar{a} G\Delta_{\overline{\varphi}}\phi-\frac{1}{2}G\frac{\partial \phi}{\partial t}-\bar{a}\dfrac{\|\nabla\phi\|^{2}}{\phi}G\\ 
&\quad-\bar{a}\dfrac{\phi(p-q+h)}{p(q-h)}\langle\nabla h,\nabla G\rangle-\bar{a}\dfrac{(p-q+h)}{p(q-h)}\langle\nabla h, \nabla\phi\rangle G.
\end{split}
\end{equation}

Now, let $(x_1, t_1)$ be a maximum point for the function $\phi G$ in the set $\mathcal{Q}_{R,T}$. If $(\phi G)\left(x_1, t_1\right) \leqslant 0$, then $(\phi G)(x, \tau) \leqslant 0$ for all $x \in M$ such that $r(x,\tau) \leqslant R$. Note that $\phi(x, \tau) \equiv 1$ for all $x \in M$ satisfying $r(x,\tau)\leqslant \frac{R}{2}$. This implies that $G(x, \tau) \leqslant 0$ when $r(x,\tau) \leqslant \frac{R}{2}$. Since $\tau$ is arbitrary, it follows that \eqref{gradient-estimates} holds on $\mathcal{Q}_{\frac{R}{2}, T}$. 

On the other hand, assume that $(\phi G)(x_1, t_1)>0$. We first address the possible lack of smoothness of the distance function at $(x_1, t_1)$. If $x_{1}\notin \textup{Cut}_{g(t_{1})}(x_{0})$ then $r$, and hence $\phi$, is smooth in a neighborhood of $(x_1, t_1)$, and the usual maximum principle applies directly. Now, suppose that $x_{1}\in \textup{Cut}_{g(t_{1})}(x_{0})$. Let $\gamma:[0,r(x_{1},t_{1})]\rightarrow M$, be a minimizing geodesic associated with the metric $g(t_{1})$ joining $x_{0}$ to $x_{1}$. For sufficiently small $\varepsilon>0$, set $x_{\varepsilon}=\gamma(\epsilon)$ and define
$$
r_{\varepsilon}(x,t)=\textup{dist}_{t}(x_{0},x_{\varepsilon})+\textup{dist}_{t}(x_{\varepsilon},x).
$$
By the triangle inequality, $r(x,t)\leqslant r_{\varepsilon}(x,t)$, while $r_{\varepsilon}(x_{1},t_{1})=r(x_{1},t_{1})$. By the standard Calabi regularization argument~\cite{calabi1958extension}, $r_{\varepsilon}$ is smooth in a neighborhood of $(x_{1},t_{1})$. Setting $\phi_{\varepsilon}(x,t)=\overline{\phi}(r_{\varepsilon}(x,t),t)$ and recalling that $\overline{\phi}_{r}\leqslant0$, we obtain $\phi_{\varepsilon}\leqslant\phi$ and $\phi_{\varepsilon}(x_{1},t_{1})=\phi(x_{1},t_{1})$. Hence $\phi_{\varepsilon}G$ also attains a local maximum at $(x_{1},t_{1})$. Therefore, all the spatial differentiations below may be performed with $\phi_{\varepsilon}$, and the corresponding estimates for $\phi$ are recovered by letting $\varepsilon\to0$. For simplicity, we continue to write $\phi$. Consequently, at $(x_{1},t_{1})$
\begin{equation}\label{Laplacian-condition}
\Delta_{\overline{\varphi}}(\phi G) \leqslant 0, \quad \nabla(\phi G)=0 \quad \mbox{ and } \quad \frac{\partial}{\partial t}(\phi G) \geqslant 0 .
\end{equation}

So, combining \eqref{Eq:06}, \eqref{estimate1}, and \eqref{Laplacian-condition}, we obtain the following estimate at the point $(x_{1},t_{1})$.
\begin{equation}\label{Eq:07}
\begin{split}
\frac{\bar{a}(q-h)}{p}\phi G^{2}&\leqslant \underbrace{\left(k_{1}+k_{2}\right)\phi G}_{\textup{I}}
+\underbrace{\frac{p\phi\langle\nabla h,\nabla b\rangle}{(q-h)^{2}}}_{\textup{II}}+\underbrace{\frac{pb}{q-h}\phi G}_{\textup{III}} \\
&+\underbrace{\left(\alpha-1+\frac{p}{q-h}\right)cu^{\alpha-1}\phi G}_{\textup{IV}}+\underbrace{\dfrac{\bar{a}(p-q+h)}{p(q-h)}\langle\gradient{h},\gradient{\phi}\rangle G}_{\textup{V}}\\
&\quad+\underbrace{\bar{a}\dfrac{\norm{\gradient{\phi}}^{2}}{\phi}G}_{\textup{VI}}-\underbrace{\dfrac{1}{2}\bar{a}G\driftedlaplacian{}{\overline{\varphi}}{\phi}}_{\textup{VII}}
+\underbrace{\frac{1}{2}G\frac{\partial \phi}{\partial t}}_{\textup{VIII}}.
\end{split}
\end{equation}

To estimate the right-hand side of the above inequality, we distinguish between the two cases: $r(x_{1},t_1)\geqslant 1$ and $r(x_1,t_1)<1$.

\smallskip
\textbf{Case 1:} $r(x_{1},t_{1})\geqslant1$. The following estimates are obtained by repeated application of the Cauchy--Schwarz and Young inequalities.

\vspace{0.1cm}
\noindent\textbf{Estimating $\textup{I}$}:
\begin{equation*}
\begin{split}
\left(k_{1}+k_{2}\right)\phi G&\leqslant\dfrac{\bar{a}\delta}{16p}\phi G^{2}+c_1(\bar{a},\delta,p)k_{1}^2+c_2(\bar{a},\delta,p)k_{2}^{2}.
\end{split}
\end{equation*}
\noindent\textbf{Estimating $\textup{II}$:}
\begin{equation*}
\begin{split}
\dfrac{p\phi}{(q-h)^{2}}\langle\nabla h,\nabla b\rangle\leqslant \dfrac{p\phi}{(q-h)^{2}}\|\nabla h\|\|\nabla b\|&=\dfrac{p\phi\|\nabla b\|G^{\frac{1}{2}}}{q-h}\\
&\leqslant\dfrac{\bar{a}\delta}{16p}\phi G^{2}+c(\bar{a},\delta, p)\sup_{\mathcal{Q}_{R,T}}\|\nabla b\|^{\frac{4}{3}}.
\end{split}
\end{equation*}

\noindent\textbf{Estimating $\textup{III}$:}
\begin{equation*}
\begin{split}
\dfrac{pb}{q-h}\phi G&\leqslant \dfrac{\bar{a}\delta}{16p}\phi G^{2}+c(\bar{a},p,\delta)(b^{+})^{2}.
\end{split}
\end{equation*}

\noindent\textbf{Estimating $\textup{IV}$:}
\begin{align*}
\!\left(\!\alpha\!-\!1\!+\!\dfrac{p}{q-h}\right)\!cu^{\alpha-1}\phi G 
&\leqslant \left[\left(\alpha-1+\dfrac{p}{q-h}\right)c \right]^{+}u^{\alpha-1}\phi G\\
&\leqslant\dfrac{\bar{a}\delta}{16p}\phi G^{2}\!+\!c(\bar{a},\delta,p)\sup_{\mathcal{Q}_{R,T}}\left\{\!\left[\!\left(\!\alpha\!-\!1\!+\!\dfrac{p}{q-h}\right)\!c\right]^{+}\!\right\}^{2}\!u^{2(\alpha-1)}.
\end{align*}

\noindent \textbf{Estimating $\textup{V}$}:
\begin{align*}
\dfrac{\bar{a}(p-q+h)}{p(q-h)}\langle\nabla h,\nabla\phi\rangle G &\leqslant
\dfrac{\bar{a}|p-q+h|}{p}\norm{\gradient{\phi}}G^{\frac{3}{2}}\\
&=\left[\frac{\bar{a}}{2p}\phi(q-h)G^{2}\right]^{\frac{3}{4}}\dfrac{\bar{a}^{\frac{1}{4}}\left(\frac{4p}{3}\right)^{\frac{3}{4}}\norm{\gradient{\phi}}|p-q+h|}{p\left[\frac{2}{3}\phi(q-h)\right]^{\frac{3}{4}}}\\
&\leqslant\dfrac{\bar{a}}{2p}\phi(q-h)G^{2}+C\left(\dfrac{\norm{\gradient{\phi}}}{\phi^{\frac{3}{4}}}\right)^{4}\frac{\bar{a}|p-q+h|^{4}}{p(q-h)^{3}}\\
&\leqslant\dfrac{\bar{a}}{2p}\phi(q-h)G^{2}+\frac{C}{R^{4}}\frac{\bar{a}|p-q+h|^{4}}{p (q-h)^{3}}.
\end{align*}

\noindent\textbf{Estimating $\textup{VI}$:}
\begin{equation*}
\begin{split}
\bar{a}\dfrac{\|\nabla\phi\|^{2}}{\phi}G=\phi^{\frac{1}{2}}G\bar{a}\dfrac{\|\nabla \phi\|^{2}}{\phi^{\frac{3}{2}}}&
\leqslant\dfrac{\bar{a}\delta}{16p}\phi G^{2}+c(\bar{a},\delta,p)\left(\dfrac{\|\nabla \phi\|^{2}}{\phi^{\frac{3}{2}}}\right)^{2}\\
&
\leqslant\dfrac{\bar{a}\delta}{16p}\phi G^{2}+\dfrac{c(\bar{a},\delta,p)}{R^{4}}.
\end{split}
\end{equation*}

Now for the $\Delta_{\overline{\varphi}} \phi$ term, we use the weighted Laplacian comparison theorem due to G. Wei and W. Wylie~\cite[Theorem~3.1]{wei2009comparison}. Since $\textup{Ric}^{\overline{\varphi}}_{g(t)}\geqslant-\frac{k_{1}}{\bar{a}}g(t)$ and $1 \leqslant r(x_{1},t_{1})\leqslant R$, we have
\begin{equation*}
\Delta_{\overline{\varphi}} r \leqslant \Gamma_{{\overline{\varphi}}}+(R-1)\frac{k_1}{\bar{a}},
\end{equation*}
whenever $1 \leqslant r \leqslant R$ and  $t_{0}-T\leqslant t \leqslant t_{0}$. Here, we have set
\begin{equation*}
\Gamma_{\overline{\varphi}}=\sup _{(x, t)\in \mathcal{Q}_{R, T}}\left\{\Delta_{\overline{\varphi}} r(x, t):  r(x,t)=1\right\}.
\end{equation*}
We then estimate $\textup{VII}$ as follows.

\vspace{0.1cm}
\noindent\textbf{Estimating $\textup{VII}$:}
\begin{equation*}
    \begin{split}
-\dfrac{1}{2}\bar{a}G\Delta_{\overline{\varphi}}\phi
&=-\dfrac{1}{2}\bar{a}G\left[\overline{\phi}' \Delta_{\overline{\varphi}}r+\overline{\phi}''
\|\nabla r\|^{2}\right]\\
&\leqslant-\dfrac{1}{2}\bar{a}G\left\{\overline{\phi}'\big{[}\Gamma_{\overline{\varphi}}+
\frac{k_{1}}{\bar{a}}(R-1)\big{]}+
\overline{\phi}''\right\}
\\
&\leqslant \bar{a}G\left\{|\overline{\phi}'|[(\Gamma_{\overline{\varphi}})^{+}+
\frac{k_{1}}{\bar{a}}(R-1)]+
|\overline{\phi}''|\right\}\\
&\leqslant \phi^{\frac{1}{2}} G\bar{a}\left\{\dfrac{|\overline{\phi}'|}{\phi^{\frac{1}{2}}}[(\Gamma_{\overline{\varphi}})^{+}+
\frac{k_{1}}{\bar{a}}(R-1)]\right\}+
\phi^{\frac{1}{2}}G\bar{a}\dfrac{|\overline{\phi}''|}{\phi^{\frac{1}{2}}}\\
&\leqslant\dfrac{\bar{a}\delta}{16p}\phi G^{2}+c(\bar{a},\delta,p)\dfrac{|\overline{\phi}''|^{2}}{\phi}+c(\bar{a},\delta,p)[(\Gamma_{\overline{\varphi}})^{+}]^{2}\dfrac{|\overline{\phi}'|^{2}}{\phi}\\
&\quad+c(\bar{a},\delta,p)k_{1}^{2}(R-1)^{2}\dfrac{|\overline{\phi}'|^{2}}{\phi}\\
&\leqslant\dfrac{\bar{a}\delta}{16p}\phi G^{2}+\dfrac{c(\bar{a},\delta,p)}{R^{4}}+\dfrac{c(\bar{a},\delta,p)[(\Gamma_{\overline{\varphi}})^{+}]^{2}}{R^{2}}+c( \bar{a},\delta,p)k_{1}^{2}.
\end{split}
\end{equation*}

\vspace{0.1cm}
\noindent\textbf{Estimating $\textup{VIII}$:}
\begin{equation*}
\begin{split}
\frac{1}{2}G\frac{\partial\phi}{\partial t}
&\leqslant G\left(\left|\frac{\partial\bar{\phi}}{\partial t}\right|+k_{2}R \left|\frac{\partial\bar{\phi}}{\partial r}\right|\right)\\
&\leqslant G\left(\frac{\left|\frac{\partial\bar{\phi}}{\partial t}\right|}{\bar{\phi}^{\frac{1}{2}}}+R \frac{k_{2}\left|\frac{\partial\bar{\phi}}{\partial r}\right|}{\bar{\phi}^{\frac{1}{2}}}\right) \bar{\phi}^{\frac{1}{2}} \\
&\leqslant  \frac{c\left[1+(\tau-t_{0}+T)k_{2}\right]}{\tau-t_{0}+T} \bar{\phi}^{\frac{1}{2}}G\\
&\leqslant\frac{\bar{a}\delta}{16p}\phi G^2+c(\bar{a},\delta,p)\left[\frac{1}{(\tau-t_{0}+T)^2}+k_{2}^{2}\right].
\end{split}
\end{equation*}
Next, we justify the estimate for the time derivative of $r(x,t)$ in the previous inequality. Fix $x \in M$ and $t \in\left[t_0-T, t_0\right]$ such that $r(x, t)<R$. From the assumption
$$
\frac{\partial}{\partial t} g(t) \geqslant-2 k_2 g(t)
$$
it follows that, for every $s>0$ such that $t+s \in\left[t_0-T, t_0\right]$, we have $g(t+s) \geq e^{-2 k_2 s} g(t)$.
Hence, for every piecewise smooth curve $\omega:[0,1] \rightarrow M$ joining $x_0$ to $x$,
$$
\int_0^1 \sqrt{g(t+s)\left(\omega^{\prime}(\tau), \omega^{\prime}(\tau)\right)} d \tau \geqslant e^{-k_2 s} \int_0^1 \sqrt{g(t)\left(\omega^{\prime}(\tau), \omega^{\prime}(\tau)\right)} d \tau .
$$
Taking the infimum over all such curves, we obtain
$$
r(x, t+s)=\textup{dist}_{t+s}\left(x, x_0\right) \geqslant e^{-k_2 s} \textup{dist}_t\left(x, x_0\right)=e^{-k_2 s} r(x, t).
$$
Therefore,
$$
\frac{r(x, t+s)-r(x, t)}{s} \geqslant \frac{e^{-k_2 s}-1}{s} r(x, t).
$$
At every point where $r$ is smooth, letting $s \rightarrow 0^{+}$ gives
$$
\frac{\partial r}{\partial t}(x, t) \geqslant-k_2 r(x, t) \geqslant-k_2 R.
$$
Notice that this estimate follows directly from the evolution inequality for the metric and, in particular, does not depend on the choice or the time variation of a minimizing geodesic.

If the maximum point $\left(x_1, t_1\right)$ lies in the cut locus, we use the regularized distance $r_{\varepsilon}$ introduced above. Applying the same metric comparison to the two distance terms in
$$
r_{\varepsilon}(x, t)=\textup{dist}_t\left(x_0, x_{\varepsilon}\right)+\textup{dist}_t\left(x_{\varepsilon}, x\right),
$$
we obtain
$$
r_{\varepsilon}(x, t+s) \geqslant e^{-k_2 s} r_{\varepsilon}(x, t).
$$
Since $r_{\varepsilon}$ is smooth in a neighborhood of $\left(x_1, t_1\right)$, it follows that
$$
\frac{\partial r_{\varepsilon}}{\partial t}\left(x_1, t_1\right) \geqslant-k_2 r_{\varepsilon}\left(x_1, t_1\right).
$$
Moreover, by the choice of $x_{\varepsilon}$,
$$
r_{\varepsilon}\left(x_1, t_1\right)=\operatorname{dist}_{t_1}\left(x_0, x_{\varepsilon}\right)+\operatorname{dist}_{t_1}\left(x_{\varepsilon}, x_1\right)=r\left(x_1, t_1\right),
$$
and hence
\[
\frac{\partial r_{\varepsilon}}{\partial t}\left(x_1, t_1\right) \geqslant-k_2 r\left(x_1, t_1\right) \geqslant-k_2 R.
\]
Thus, the same time derivative estimate remains valid under the Calabi regularization at the maximum point.
 
Therefore, at $(x_{1}, t_{1})$, we have the following estimate
\begin{equation*}
\begin{split}
\phi G^{2}&\leqslant c(\bar{a},\delta,p)k_{1}^{2}+c(\bar{a},\delta,p)k_{2}^{2}+\frac{c(\bar{a},\delta,p)}{R^{4}}+\frac{c(\bar{a},\delta,p)}{(\tau-t_{0}+T)^{2}}+\frac{c(\bar{a},\delta,p)[(\Gamma_{\overline{\varphi}})^{+}]^2}{R^2}\\
&\quad+c(\bar{a},\delta,p)\left(\sup_{\mathcal{Q}_{R,T}}b^{+}\right)^{2}+c(\delta,p)\sup_{\mathcal{Q}_{R,T}}\left\{\left[\left(\alpha-1+\dfrac{p}{q-h}\right)c\right]^{+}\right\}^{2}u^{2(\alpha-1)}\\
&\quad+c(\bar{a},\delta, p)\sup_{\mathcal{Q}_{R,T}}\norm{\gradient{b}}^{\frac{4}{3}},
\end{split}
\end{equation*}
where we have used that
\begin{equation*}
\frac{|p-q+h|}{q-h}\leqslant\frac{p+|h-q|}{q-h}\leqslant\frac{p}{q-h}+1\leqslant\frac{p}{\delta}+1.
\end{equation*}
So,
\begin{equation*}
\begin{split}
    (\phi^{2} G^{2})(x_{1},t_{1})&\leqslant (\phi G^{2})(x_{1},t_{1})\\
&\leqslant c(\bar{a},\delta,p)k_{1}^{2}+c(\bar{a},\delta,p)k_{2}^{2}+\frac{c(\bar{a},\delta,p)}{R^{4}}+\frac{c(\bar{a},\delta,p)}{(\tau-t_{0}+T)^{2}}\\
&\quad+c(\bar{a},\delta,p)\left(\sup_{\mathcal{Q}_{R,T}}b^{+}\right)^{2}+c(\bar{a},\delta, p)\sup_{\mathcal{Q}_{R,T}}\|\nabla b\|^{\frac{4}{3}}+\frac{c(\bar{a},\delta,p)[(\Gamma_{\overline{\varphi}})^{+}]^2}{R^2}\\
&\quad+c(\delta,p)\sup_{\mathcal{Q}_{R,T}}\left\{\left[\left(\alpha-1+\dfrac{p}{q-h}\right)c\right]^{+}\right\}^{2}u^{2(\alpha-1)}.
\end{split}
\end{equation*}
Since $\phi(x,\tau)=1$ when $r(x,\tau)<\frac{R}{2}$, it follows that
\begin{equation*}
\begin{split}
G(x,\tau)&= (\phi G)(x,\tau)\\
&\leqslant (\phi G)(x_{1},t_{1})\\
& \leqslant c(\bar{a},\delta,p)k_{1}+c(\bar{a},\delta,p)k_{2}+\frac{c(\bar{a},\delta,p)}{R^{2}}+\frac{c(\bar{a},\delta,p)}{\tau-t_{0}+T}\\
&\quad+c(\bar{a},\delta,p) \sup_{\mathcal{Q}_{R,T}} b^{+}+c(\delta,p)\sup_{\mathcal{Q}_{R,T}}\left\{\left[\left(\alpha-1+\dfrac{p}{q-h}\right)c\right]^{+}\right\}u^{(\alpha-1)}\\
&\quad+c(\bar{a},\delta, p)\sup_{\mathcal{Q}_{R,T}}\|\nabla b\|^{\frac{2}{3}}+\frac{c(\bar{a},\delta,p)[(\Gamma_{\overline{\varphi}})^{+}]}{R}
\end{split}
\end{equation*}
that implies 
\begin{equation*}
\begin{split}
\dfrac{\norm{\gradient{h}}}{q-h}(x,\tau)& \leqslant c(\bar{a},\delta,p)k_{1}^{\frac{1}{2}}+c(\bar{a},\delta,p)k_{2}^{\frac{1}{2}}+\frac{c(\bar{a},\delta,p)}{R}+\frac{c(\bar{a},\delta,p)}{(\tau-t_{0}+T)^{\frac{1}{2}}}\\
&\quad+\frac{c(\bar{a},\delta,p)[(\Gamma_{\overline{\varphi}})^{+}]^{1/2}}{\sqrt{R}}+c(\bar{a},\delta,p)\left(\sup_{\mathcal{Q}_{R,T}}b^{+}\right)^{\frac{1}{2}}+c(\bar{a},\delta, p)\sup_{\mathcal{Q}_{R,T}}\|\nabla b\|^{\frac{1}{3}}\\
&\quad+c(\delta,p)\sup_{\mathcal{Q}_{R,T}}\left\{\left[\left(\alpha-1+\dfrac{p}{q-h}\right)c\right]^{+}\right\}^{1/2}u^{\frac{\alpha-1}{2}}
\end{split}
\end{equation*}
for all $x\in M$ such that $r(x,\tau) \leqslant R/2$. Since $h = p\ln u$, substituting this into the above estimate completes the proof in this case.

\textbf{Case 2:} $r(x_{1},t_{1})<1$. In this case, $\phi$ is a constant in space direction in $\mathcal{Q}_{\frac{R}{2}, T}$ where $R \geqslant 2$. Thus, from \eqref{Eq:07}, we have
\begin{equation*}
\begin{split}
\frac{\bar{a}(q-h)}{p}\phi G^{2}&\leqslant \underbrace{\left(k_{1}+k_{2}\right)\phi G}_{\textup{I}}
+\underbrace{\frac{p\phi\langle\nabla h,\nabla b\rangle}{(q-h)^{2}}}_{\textup{II}}+\underbrace{\frac{pb}{q-h}\phi G}_{\textup{III}} \\
&+\underbrace{\left(\alpha-1+\frac{p}{q-h}\right)cu^{\alpha-1}\phi G}_{\textup{IV}}\\
&\leqslant\dfrac{\bar{a}\delta}{4p}\phi G^{2}+c_1(\bar{a},\delta,p)k_{1}^2+c_2(\bar{a},\delta,p)k_{2}^{2}+c(\bar{a},\delta, p)\sup_{\mathcal{Q}_{R,T}}\|\nabla b\|^{\frac{4}{3}}\\
&\quad +c(\bar{a},p,\delta)\!\left(\sup_{\mathcal{Q}_{R,T}}b^{+}\!\right)^{2}\!\!\!+\!c(\bar{a},\delta,p)\sup_{\mathcal{Q}_{R,T}}\!\left\{\!\left[\!\left(\!\alpha\!-\!1\!+\!\dfrac{p}{q-h}\right)\!c\right]^{+}\!\right\}^{2}\!\!\!u^{2(\alpha-1)}.
\end{split}
\end{equation*}
So,
\begin{equation*}
\begin{split}
\phi G^{2}&\leqslant c_1(\bar{a},\delta,p)k_{1}^2+c_2(\bar{a},\delta,p)k_{2}^{2}+c(\bar{a},\delta, p)\sup_{\mathcal{Q}_{R,T}}\|\nabla b\|^{\frac{4}{3}}+c(\bar{a},p,\delta)\left(\sup_{\mathcal{Q}_{R,T}} b^{+}\right)^{2}\\
&\quad +c(\bar{a},\delta,p)\sup_{\mathcal{Q}_{R,T}}\left\{\left[\left(\alpha-1+\dfrac{p}{q-h}\right)c\right]^{+}\right\}^{2}\!u^{2(\alpha-1)}.
\end{split}
\end{equation*}
Recall that $\phi(x,\tau)=1$ when $r(x,\tau)<R/2$. So,
\begin{align*}
G(x,\tau)&= (\phi G)(x,\tau)
\leqslant (\phi G)(x_{1},t_{1})\\
&\leqslant c_1(\bar{a},\delta,p)k_{1}+c_2(\bar{a},\delta,p)k_{2}+c(\bar{a},\delta, p)\sup_{\mathcal{Q}_{R,T}}\|\nabla b\|^{\frac{2}{3}}+c(\bar{a},p,\delta)\sup_{\mathcal{Q}_{R,T}} b^{+}\\
&\quad +c(\bar{a},\delta,p)\sup_{\mathcal{Q}_{R,T}}\left\{\left[\left(\alpha-1+\dfrac{p}{q-h}\right)c\right]^{+}\right\}u^{\alpha-1}.
\end{align*}
for all $x\in M$ such that $r(x,\tau)<R/2$. By the definition of $G(x,\tau)$, we prove that the estimate in the theorem still holds when $r(x_{1},t_{1})<1$.
\end{proof}

\begin{proof}[\bf Proof of Theorem~\ref{theorem**}]  From Theorem~\ref{thm1}, the function $u=f^{\frac{1}{\sigma}}$ provides a solution of 
\begin{equation}\label{PDE-1}
\begin{split}
\frac{\partial u}{\partial t}=\Delta_{\varphi}u-2m\rho\Delta u+\frac{\rho}{\sigma} S_{g(t)}u
+\frac{m\rho-1}{m\sigma}S_{F
}u^{1-2\sigma}
\end{split}
\end{equation}
on $B^n$. 

We first prove item (b). Suppose that $\overline{g}(t)$ is an ancient or eternal solution with $S_{F}<0$ and $\rho S_{g(t)}<0$. Fix an arbitrary point $(x_{0},t_{0})\in B\times I$, and apply Theorem~\ref{main} to $u(x_{0},t_{0})$ in the set $\mathcal{Q}_{R/2,R}$ with the parameters
\begin{equation*}
a=2m\rho,\ \ b= \frac{\rho}{\sigma}S_{g(t)},
\ \ c=\frac{m\rho-1}{m\sigma}S_{F}, \ \ \alpha=1-2\sigma, \ \ p=1,\ \ q=\frac{1}{\sigma}+\ln D(\mathcal{Q}_{R,R}).
\end{equation*}
This yields
\begin{equation*}
\begin{split}
&\|\nabla \ln u(x_{0},t_{0})\|\leqslant C\left(\frac{1}{\sigma}+\ln D(\mathcal{Q}_{R,R})-\ln u(x_{0},t_{0})\right)\Bigg{\{}\dfrac{1}{R}\!+\!\dfrac{1}{\sqrt{R}}\!+\!\frac{[(\Gamma_{\overline{\varphi}})^{+}]^{\frac{1}{2}}}{R^{\frac{1}{2}}}\\
&\quad\qquad+\sqrt{b^{+}}\!+\!\sup_{\mathcal{Q}_{R,R}}\|\nabla b\|^{\frac{1}{3}}+\sup_{\mathcal{Q}_{R,R}}\sqrt{\left[\left(\alpha-1+\dfrac{p}{q-p\ln u(x_1,t_1)}\right)c\right]^{+}}u^{\frac{\alpha-1}{2}}\Bigg{\}}
\end{split}
\end{equation*}
where $(x_{1},t_{1})$ is a maximum point of $\phi G$  as in the proof of Theorem~\ref{main}. 

Note that $\rho<\frac{1}{2m}\leqslant\frac{1}{m+1}<\frac{1}{m}$. Hence, $\sigma>0$ and then $b<0$ and $c>0$. Moreover
\begin{equation*}
\begin{split}
  \alpha-1+\dfrac{p}{q-p\ln u(x_{1},t_{1})}=-2\sigma+\dfrac{1}{\frac{1}{\sigma}+\ln\left[\frac{D(\mathcal{Q}_{R,R})}{u(x_{1},t_{1})}\right]}\leqslant-\sigma\leqslant0.
\end{split}
\end{equation*}
Since $c>0$, one has
\begin{equation*}\left[\left(\alpha-1+\dfrac{p}{q-p\ln u(x_{1},t_{1})}\right)c\right]^{+}=0.
\end{equation*}
Thus
\begin{equation*}
\begin{split}
\|\nabla\ln u(x_{0},t_{0})\|&\leqslant C\left(\frac{1}{\sigma}+\ln D(\mathcal{Q}_{R,R})-\ln u(x_{0},t_{0})\right)\Bigg{\{}\dfrac{1+[(\Gamma_{\overline{\varphi}})^{+}]^{\frac{1}{2}}}{R^{\frac{1}{2}}}\Bigg{\}}.
\end{split}
\end{equation*}
Since $\|\nabla \varphi\|_{g(t)}\leqslant C_{0}$ for some constant $C_{0}$, we deduce from the Laplacian comparison theorem that 
\[
\Delta_{\bar{\varphi}} r(x,t)
\leqslant\frac{n-1}{r}+\frac{C_{0}}{1-a}.
\]
Therefore, along the geodesic sphere $r(x,t)=1$, one may take
\[
\Gamma_{\bar{\varphi}}
\leqslant n-1+\frac{C_{0}}{1-a},
\]
which shows that $\Gamma_{\bar{\varphi}}$ can be chosen uniformly in $t$, and $\Gamma_{\overline{\varphi}}<\infty$.

Now, the growth condition $\ln f(x,t)=o\left(r(x,t)^{\frac{1}{2}}+|t|^{\frac{1}{2}}\right)$ near infinity implies the analogous bound for $\ln u(x,t)$. Hence,
\begin{equation*}
\begin{split}
\norm{\gradient{\ln u(x_{0},t_{0})}}&\leqslant C\left(\frac{1}{\sigma}+o(R^{\frac{1}{2}})-\ln u(x_{0},t_{0})\right)\Bigg{\{}\dfrac{1+[(\Gamma_{\overline{\varphi}})^{+}]^{\frac{1}{2}}}{R^{\frac{1}{2}}}\Bigg{\}}.
\end{split}
\end{equation*}
Letting $R\to+\infty$ forces $\norm{\gradient{\ln u(x_{0},t_{0})}}=0$. Since the initial choice of $(x_{0},t_{0})$ was arbitrary, it follows that $u(x,t)=u(t)$ for all $x\in B^n$. Hence, from \eqref{PDE-1}, we obtain the ordinary differential equation
\begin{equation*}
u'(t)=\frac{\rho}{\sigma}S_{g(t)}u(t)+\frac{m\rho-1}{m\sigma}S_{F}u(t)^{1-2\sigma}.
\end{equation*}
In particular, when $S_{g(t)}=S_{0}$ is constant in time, integration gives
\begin{equation*}
u(t)^{2\sigma}=Ce^{2\rho S_{0} t}-\left(1-\frac{1}{m\rho}\right)\frac{S_{F}}{S_{0}}.
\end{equation*}
for some constant $C$. We claim that $C=0$. Indeed, if $C\neq 0$, then the exponential term forces $u$ to violate the growth condition when $t\to-\infty$. Hence $u(x,t)$ must be constant, and consequently $f(x,t)$ is constant as well.

Now, for ancient or eternal solution with $S_{F}=0$ and $\rho S_{g(t)}=0$, we obtain the following estimate:
\begin{equation*}
\begin{split}
\norm{\gradient{\ln u(x_{0},t_{0})}}&\leqslant C\left(\frac{1}{\sigma}+o(R^{\frac{1}{2}})-\ln u(x_{0},t_{0})\right)\Bigg{\{}\dfrac{1+[(\Gamma_{\overline{\varphi}})^{+}]^{\frac{1}{2}}}{R^{\frac{1}{2}}}\Bigg{\}}.
\end{split}
\end{equation*}
Letting $R\to+\infty$ yields $\norm{\gradient{\ln u(x_{0},t_{0})}}=0$, and hence
$u(x,t)=u(t)$ for all $x\in B^n$. From \eqref{PDE-1}, we obtain $u'(t)=0$ for all $t$, and hence $u$ is constant and consequently $f$ is constant as well. This completes the proof of item (b).

We now prove item (a). Suppose that $\overline{g}(t)$ is an immortal solution on $(T,\infty)$. For each sufficiently large $t_{0}>T$, set 
$$
R=\frac{t_{0}-T}{2}.
$$
Then $R\geqslant2$ and $t_{0}-R=\frac{t_{0}+T}{2}>T$. Therefore, $\mathcal{Q}_{R,R}\subset B^{n}\times (T,\infty)$. Moreover, $R\to\infty$ as $t_{0}\to \infty$. Using the same choices of $a,b,c, \alpha, p$ and $q$ together with the sign arguments established above, we obtain
\begin{equation*}
\begin{split}
\|\nabla\ln u(x_{0},t_{0})\|&\leqslant C\left(\frac{1}{\sigma}+o(R^{\frac{1}{2}})-\ln u(x_{0},t_{0})\right)\Bigg{\{}\dfrac{1+[(\Gamma_{\overline{\varphi}})^{+}]^{\frac{1}{2}}}{R^{\frac{1}{2}}}\Bigg{\}}\\
&\leqslant C\left(\frac{1}{\sigma}+o(R^{\frac{1}{2}})+|\ln u(x_{0},t_{0})|\right)\Bigg{\{}\dfrac{1+[(\Gamma_{\overline{\varphi}})^{+}]^{\frac{1}{2}}}{R^{\frac{1}{2}}}\Bigg{\}}\\
&\leqslant C\left(\frac{1}{\sigma}+o(R^{\frac{1}{2}})+o(R^{\frac{1}{2}})\right)\Bigg{\{}\dfrac{1+[(\Gamma_{\overline{\varphi}})^{+}]^{\frac{1}{2}}}{R^{\frac{1}{2}}}\Bigg{\}}
\end{split}
\end{equation*}
Letting $t_{0}\to\infty$, we have $R\to\infty$ and using $\ln f=\sigma \ln u$, we conclude that
$$
\lim_{t_{0}\to\infty}\|\nabla \ln f(x_{0},t_{0})\|_{g(t_{0})}=0.
$$
\end{proof}

\begin{proof}[\bf Proof of Theorem~\ref{theorem***}]
We proceed by contradiction. Suppose that there exists a solution satisfying the assumptions of the theorem. By Theorem~\ref{thm1}, the function $u=f^{\frac{1}{\sigma}}$ satisfies
\begin{equation}\label{flow1}
\begin{split}
\frac{\partial u}{\partial t}=\Delta_{\varphi}u-2m\rho\Delta u+\frac{\rho}{\sigma} S_{g(t)}u
+\frac{m\rho-1}{m\sigma}S_{F
}u^{1-2\sigma}.
\end{split}
\end{equation}
Analogously to the proof of Theorem~\ref{theorem**}, we can show that $\nabla \ln u(\cdot,t)=0$ if either of the sign assumptions of Theorem~\ref{theorem***} holds. Thus, there exists a positive function $U:I\rightarrow (0,\infty)$ such that $u(x,t)=U(t)$. Consequently, from~\eqref{flow1}, the function $F(t):=U(t)^{\sigma}$ satisfies
$$
\frac{d}{d t} F(t)^2=2 \rho S_{g(t)} F(t)^2+\frac{2(m \rho-1)}{m} S_F.
$$
Suppose first that $S_F=0$. Then
$$
\frac{F^{\prime}(t)}{F(t)}=\rho S_{g(t)} \leqslant-\kappa.
$$
For $t \leqslant t_0$, we have
$$
\ln F(t)=\ln F\left(t_0\right)-\int_t^{t_0} \rho S_{g(s)} d s \geqslant \ln F\left(t_0\right)+\kappa\left(t_0-t\right).
$$
Therefore, in the ancient case,
$$
\frac{\ln F(t)}{|t|^{1 / 2}} \longrightarrow+\infty \quad \text { as } t \rightarrow-\infty
$$
which again contradicts the growth assumption. The eternal case is also impossible because it contains both temporal directions.

Suppose now that $S_F<0$ and $\rho S_{g(t)}=0$. Then
$$
\frac{d}{d t} F(t)^2=\frac{2(m \rho-1)}{m} S_F.
$$
Since $1-2 m \rho>0$, we have $m \rho-1<0$. Therefore, $2(m \rho-1)S_F/m>0$. After integration,
$$
F(t)^2=F\left(t_0\right)^2+\frac{2(m \rho-1)}{m} S_F\left(t-t_0\right).
$$
As $t \rightarrow-\infty$, the right-hand side goes to $-\infty$, contradicting the positivity of $F(t)^2$. Therefore, no ancient or eternal solution exists.

In the immortal case, assuming that $S_F=0$ and $\rho S_{g(t)}=0$, an argument analogous to that used in the proof of Theorem~\ref{theorem**} yields $\lim_{t\to\infty}\|\nabla \ln u(x,t)\|=0$ for $x\in B$. Moreover, from
\begin{equation*}
    \frac{\partial^2}{\partial t^{2}}\left(\|\nabla \ln f(x,t)\|_{g(t)}^{2}\right)\leqslant0\ \ \mbox{on} \ \ B\times I
\end{equation*}
the function $t\longmapsto \|\nabla \ln u(x,t)\|_{g(t)}^{2}$ is nonnegative and concave for each fixed $x\in B$. Its convergence to zero as $t\to\infty$ therefore implies that it vanishes identically. Consequently,
 $\nabla \ln u(x,t)=0$  and then $u(x,t)=U(t)$ for some positive function $U$. Substituting this expression into \eqref{flow1}, we obtain $U'(t)=0$. Thus, $u$ is constant, and consequently $f$ is constant, which contradicts the assumptions.
\end{proof}

\section{\bf Acknowledgments}
Willian I. Tokura thanks the Department of Mathematics at the Institute of Mathematical and Computer Sciences (ICMC) for organizing the 2026 edition of the ICMC Summer Meeting on Differential Equations, during which part of this work was carried out. He is also grateful to Alessandra Verri (UFSCar) for her hospitality. José N.V. Gomes has been partially supported by Conselho Nacional de Desenvolvimento Científico e Tecnológico (CNPq), Grant 309825/2026-1, and Fundação de Amparo à Pesquisa do Estado de São Paulo (FAPESP), Grants 2024/00923-6, 2023/11126-7 and 2022/16097-2. Hikaru Yamamoto has been partially supported by JSPS KAKENHI Grant-in-Aid for Early-Career Scientists 22K13909.

\bibliography{ref/main}

\bibliographystyle{plain}

\end{document}